\documentclass[11pt]{article}

\usepackage[T1]{fontenc}
\usepackage{lmodern}
\usepackage[letterpaper,margin=1in]{geometry}
\usepackage{amsmath,amssymb,amsthm,mathtools}
\usepackage{booktabs,array,tabularx,graphicx,microtype,float}
\usepackage{tikz}
\usepackage{enumitem}
\usepackage[hidelinks]{hyperref}
\usepackage[nameinlink,capitalise]{cleveref}

\setlist{leftmargin=1.5em,itemsep=0.15em,topsep=0.25em}
\graphicspath{{./}{figures/}}

\newcommand{\E}{\mathbb E}
\newcommand{\Pp}{\mathbb P}
\newcommand{\Var}{\operatorname{Var}}
\newcommand{\Cov}{\operatorname{Cov}}
\newcommand{\Law}{\operatorname{Law}}
\newcommand{\dTV}{d_{\mathrm{TV}}}

\newcommand{\argmaxx}{\operatorname*{arg\,max}}
\newcommand{\Nplus}{\mathbb N_{+}}
\newcommand{\dd}{\,\mathrm d}

\newtheoremstyle{assumptionstyle}
  {12pt plus 2pt minus 2pt} 
  {12pt plus 2pt minus 2pt} 
  {\normalfont}            
  {}                       
  {\bfseries}              
  {.}                      
  {\newline}               
  {}                       
\theoremstyle{assumptionstyle}
\newtheorem{assumption}{Assumption}
\crefname{assumption}{assumption}{assumptions}
\Crefname{assumption}{Assumption}{Assumptions}
\theoremstyle{plain}

\newtheorem{theorem}{Theorem}
\newtheorem*{informaltheorem}{Theorem}

\newtheorem{lemma}{Lemma}
\newtheorem{corollary}{Corollary}

\theoremstyle{definition}

\theoremstyle{remark}
\newtheorem{remark}{Remark}

\title{Optimal support and condensation in random allocations}
\author{Andrea Ottolini}
\date{}

\begin{document}
\maketitle

\begin{abstract}
How many distinct symbols should a password use?  If its length is fixed at $n$ and an observer learns only which symbols appear, the number of compatible passwords is maximized asymptotically when $k/n\to1/(2\log2)$.  We ask what changes when, in addition to the length, aggregate information about the repetition pattern is revealed.  We model this by fixing a second additive profile $V_k=\sum_i v(J_i)$ at scale $V_k/n\approx\rho$.  For $v(j)=\log(j!)$, the profile records the reduction in the logarithm of the number of compatible words caused by repetitions; we show that once the normalized profile $\rho$ exceeds $0.507834\ldots$, the limiting optimal fraction is pinned at $1/2$.  For a typical multiplicity profile at the optimal support above this threshold, the excess in $V_k$ is carried by a vanishing fraction of used symbols.  We interpret this as a form of non-equivalence of ensembles and extend the mechanism to other profiles and non-uniform allocation models.
\end{abstract}

\noindent\textbf{Keywords.}
Random allocations; canonical and microcanonical ensembles; conditional limit theory; equivalence of ensembles; entropy optimization; condensation.

\section{Introduction}
\label{sec:intro}

Suppose a user enters a password with greasy fingers, leaving oily residue on the buttons that were touched. Such traces can leak information through so-called ``smudge attacks''~\cite{AvivSmudge}. Assume that the adversary knows the password length $n$ and learns from the residue only which buttons were used. If exactly $k$ buttons are marked, there are $k!S(n,k)$ compatible passwords, where $S(n,k)$ is a Stirling number of the second kind. It is then natural to design a password to maximize $k!S(n,k)$. It is standard that a maximizer $k_n$ satisfies $k_n/n\to1/(2\log2)$; see Mez\H{o}~\cite{Mezo} and Temme~\cite{Temme}.

We analyze variants in which, beyond the length, other exchangeable features of the password are assumed to be known. Our main Theorems~\ref{thm:two-canonical} and~\ref{thm:two-supercritical} cover a broad class of additive statistics $V_k=\sum_i v(J_i)$, where $J_1,\ldots,J_k$ are the symbol multiplicities, and also allow non-uniform weights. Under the constraints $T_k=\sum_iJ_i=n$ and $V_k\approx\rho n$, we ask: what is the optimal number of symbols to use, and what does a typical optimal password look like? As an illustration of these general results, consider the uniform model with $v(j)=\log(j!)$. Notice that $\log(n!)-V_k$ is the logarithm of the number of distinct rearrangements of a password. Thus $\rho$ provides a natural measure of how ``simple'' a password is in terms of its symbol multiplicities: at fixed length, larger $\rho$ means fewer distinct rearrangements, reflecting greater repetition. The limiting optimal fraction has the following behavior; see Corollary~\ref{cor:typeclass} for the precise statement.

\begin{informaltheorem}[Informal]
For $\rho\ge0$, let $k_n^*(\rho)$ maximize the number of length-$n$ passwords using exactly $k$ symbols and satisfying $V_k\approx\rho n$.  There is a constant $\rho^*=0.507834\ldots$ and a limiting optimal fraction $\alpha^*(\rho)$ such that $k_n^*(\rho)/n\to\alpha^*(\rho)$.  As $\rho$ increases from $0$ to $\rho^*$, the function $\alpha^*(\rho)$ decreases strictly from $1$ to $1/2$; for every $\rho\ge\rho^*$ it remains pinned at $\alpha^*(\rho)=1/2$.
\end{informaltheorem}

The transition can be understood as follows. For the factorial profile, each multiplicity vector contributes $n!e^{-V_k}$ words. On the constraint window, these weights differ
only by a bounded factor. After removing the common factor $n!e^{-\rho n}$, counting words and counting admissible multiplicity vectors therefore agree at the leading exponential order.
The unrestricted number of positive compositions,
$\binom{n-1}{k-1}$, is maximized at $k/n\to1/2$.  Under a uniform composition at this optimum, the multiplicity of a uniformly chosen used symbol is approximately geometric with parameter $1/2$: if $J$ denotes this limiting multiplicity, then $\Pp(J=r)=2^{-r}$ for $r\ge1$.  Since there are asymptotically $n/2$ used symbols, the resulting value of $V_k/n$ is therefore $\frac12\E\log(J!)$.  Expanding $\log(J!)=\sum_{m=1}^{J}\log m$ gives
\[
   \frac12\E\log(J!)
   =\frac12\sum_{m\ge1}\Pp(J\ge m)\log m
   =\sum_{m\ge1}2^{-m}\log m
   =0.507834\ldots=: \rho^*.
\]
Above $\rho^*$, the excess profile can be carried by a vanishing fraction of symbols without changing the limiting geometric bulk or support fraction. This gives the pinning mechanism.
\begin{figure}[H]
\centering
\includegraphics[width=.88\textwidth]{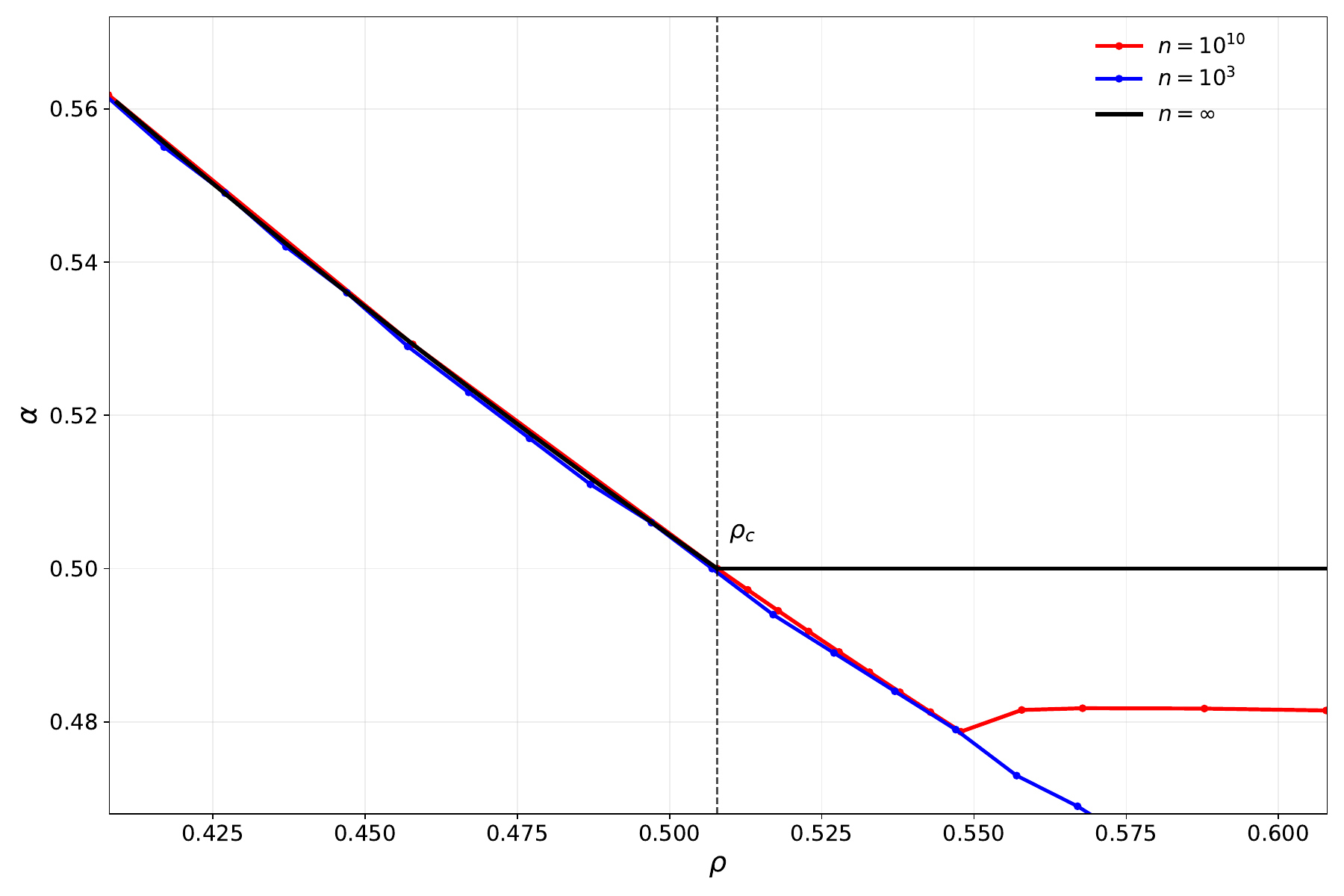}
\caption{The black curve, labeled $n=\infty$, is the limiting optimizer $\alpha^*(\rho)$, and the blue curve shows the optimizer for $n=10^3$. The red curve shows approximate values for $n=10^{10}$, accurate to within $5\times10^{-4}$ at the plotted points. The figure illustrates the slower finite-size approach to the pinned regime.}
\label{fig:certified-bounds}
\end{figure}

To explain the subcritical regime, we compare the microcanonical and canonical ensembles. In the terminology of statistical physics, a \emph{microcanonical ensemble} is supported on configurations satisfying prescribed constraints (here, an exact length and a fixed-width profile window), whereas a \emph{canonical ensemble} matches them in expectation. The phrase \emph{equivalence of ensembles} refers, roughly, to situations in which these descriptions have the same leading-order behavior.
First, revisit the introductory example. Let $J_1,\ldots,J_k$ be independent positive-Poisson variables with law $q_x(j)=x^j/[j!(e^x-1)]$, $j\ge1$. Then
\begin{equation}
   \Pp_x(T_k=n)
   =\frac{x^n}{(e^x-1)^k}\frac{k!S(n,k)}{n!}
   \label{eq:poisson-stirling}
\end{equation}
Indeed, after conditioning on $T_k=n$, the factor $x^{J_1+\cdots+J_k}$ is common to every composition, so the induced profile weight is proportional to $1/(J_1!\cdots J_k!)$.  Up to the common factor $n!$, this is exactly the number of words with the corresponding multiplicities.  Thus the multiplicities of a uniformly chosen password of length $n$ using $k$ symbols have the law of $k$ i.i.d. positive-Poisson variables conditioned on their sum. More precisely, choose $x=\log2$, so that $e^x-1=1$.  In \eqref{eq:poisson-stirling} the factor outside the probability is then independent of $k$, and maximizing $k!S(n,k)$ is exactly the same as maximizing $\Pp_x(T_k=n)$ for one fixed positive-Poisson law.  The latter is maximized asymptotically when the sum is centered, $k\E_xJ\sim n$.  Since $\E_xJ=xe^x/(e^x-1)=2\log2$, this gives $k/n\to1/(2\log2)$. Theorem~\ref{thm:one-stat} extends this argument to non-uniform weights.

The factorial profile $V_k=\sum_i\log(J_i!)$ leads to the law $q_{x,\gamma}(j)=x^j(j!)^{\gamma-1}/Z_\gamma(x)$.  Its normalizing function is $Z_\gamma(x)=\sum_{j\ge1}x^j(j!)^{\gamma-1}$.  This is the positive Conway--Maxwell--Poisson family~\cite{ConwayMaxwell} after the conventional reparametrization $\nu=1-\gamma$.  Its sufficient statistics are $J$ and $\log(J!)$.  If $k/n\to\alpha$, matching the two microcanonical constraints means choosing $(x,\gamma)$ so that a multiplicity of a single used symbol has mean $1/\alpha$ and expected value of $\log(J!)$ equal to $\rho/\alpha$.  Thus the two parameters are used to match, in expectation, precisely the two quantities that are fixed in the microcanonical problem.

The difference between the two ensembles becomes visible at the edge of this family.  At $\gamma=0$ we recover the positive-Poisson model, and increasing $\gamma$ increases the canonical mean of $\log(J!)$.  At $\gamma=1$, after conditioning on the total, every positive composition of $n$ into $k$ parts has the same probability.  The support optimization at this boundary gives $\alpha=1/2$ and $V_k/n=\rho^*$ asymptotically.  For $\rho>\rho^*$, no parameter in this family matches both required means at the limiting optimal support; indeed the normalizing series ceases to be finite once $\gamma>1$.

The canonical description also identifies the multiplicities of typical used symbols.  The following statement is informal; see Corollary~\ref{cor:typeclass} for the precise quantitative formulation.
\begin{informaltheorem}[Informal]
Consider the uniform model with the factorial profile $V_k\approx\rho n$ at its optimizing support.
\begin{itemize}
\item If $0<\rho<\rho^*$, let $(x,\gamma)$ be the canonical parameters matching the two constraints.  For every fixed $r$, the multiplicities of any fixed $r$ used symbols converge in total variation to $q_{x,\gamma}^{\otimes r}$.
\item If $\rho>\rho^*$, then $k/n\to1/2$, and every fixed collection of used symbols is asymptotically described by the boundary canonical law.  The typical bulk accounts for profile rate $\rho^*$, while the excess profile rate $\rho-\rho^*$ is carried by a vanishing fraction of used symbols with atypically large multiplicities.
\end{itemize}
\end{informaltheorem}

Here non-equivalence refers to the failure of canonical moment matching at the optimal support: above $\rho^*$, the boundary law does not reproduce the prescribed profile rate. Equivalence of finite marginals and vanishing relative entropy per symbol persist.

\subsection*{Relation with existing work}
\label{sec:literature}

The support-only problem is classical.  The count $k!S(n,k)$ is an ordered Stirling number, and its asymptotic mode has been studied by Mez\H{o}~\cite{Mezo} and Temme~\cite{Temme}.  More generally, our fixed-support partition function is a partial Bell polynomial: if $w_j=j!a_j$, then $W^{(a)}_{n,k}=k!B_{n,k}(w)/n!$.  Thus the fixed-$k$ measures belong to the weighted-partition framework discussed by Pitman and by Berestycki--Pitman~\cite{PitmanCSP,BerestyckiPitman}.  The profile constraint refines this ensemble by fixing an additional additive statistic of the block sizes.  Models with a varying number of boxes were studied by Bia\l as--Burda--Johnston~\cite{BialasBurdaJohnston}. Compositions also admit a renewal interpretation, with $k$ as the renewal count and $V_k$ as a cumulative reward; see Zamparo~\cite{ZamparoRenewal,ZamparoModels}. Here we optimize $k$ under an exact length constraint and a fixed-width constraint on a superlinear profile, and obtain quantitative estimates for the optimizer and conditional marginals.

The positive-Poisson conditioning identity in \eqref{eq:poisson-stirling} is standard in the probabilistic treatment of surjections and occupancy counts~\cite{KolchinEtAl,CarayolRotondo}.  More broadly, equivalence of ensembles may be formulated at the level of entropy or at the level of finite marginals~\cite{LewisPfisterSullivan1994,LewisPfisterSullivan,Touchette2015,VanCampenhoutCover,Csiszar1984}.  The latter is the comparison used here: below criticality the microcanonical law of finitely many symbols approaches the corresponding canonical product law.
The one-statistic marginal approximation is already part of finite de Finetti theory.  Diaconis--Freedman~\cite{DiaconisFreedmanExp,DiaconisFreedmanDozen} derive total-variation estimates for product laws conditioned on sufficient statistics.  In the notation of Theorem~\ref{thm:one-stat}, this compares the first $r$ coordinates of $k$ multiplicities conditioned on $T_k=n$ with $q_x^{\otimes r}$, where $\E_xJ=n/k$.  Under their regularity hypotheses, the error is of order $r/k$ when $r,k\to\infty$ and $r/k\to0$.  In the supercritical regime, convergence of finite marginals persists but does not capture the excess profile carried by atypically large multiplicities.

Our results are closely related to condensation and localization in microcanonical ensembles with several unbounded constraints~\cite{Chatterjee,Nam2020}. One motivation for this theory comes from statistical approaches to the soliton resolution conjecture for the focusing nonlinear Schr\"odinger equation, where mass and energy provide two conserved quantities. Chatterjee~\cite{ChatterjeeSoliton} developed such an approach, and Nam~\cite[Section~1.1]{Nam2020} explicitly used this setting to motivate a general study of ensembles under multiple constraints. In our allocation models, the limiting bulk likewise fails to account for the full prescribed value of an unbounded observable. We study how this mechanism selects the optimal support and causes its limiting fraction to remain pinned above criticality.

Chatterjee's simplex--sphere model~\cite{Chatterjee} gives a closely related example in which typical coordinates remain asymptotically exponential while one coordinate becomes much larger than the others. Related localization results appear for weighted allocations and zero-range models~\cite{GrosskinskySchutzSpohn,ArmendarizLoulakisHeavy,JansonAllocations,SzavitsPRL,SzavitsJPA}. Under our general assumptions we prove localization on a vanishing fraction of symbols.
Related questions of ensemble equivalence arise for random graphs and random permutations~\cite{SquartiniEtAl,DenHollanderDense,KenyonKralRadinWinkler,MukherjeePermExp,BorgaDasMukherjeeWinkler}.  Finally, the equations $A(x)=1$ and $Z_\gamma(x)=1$ also have a direct generating-function interpretation.  For fixed $k$, one has $\sum_n W^{(a)}_{n,k}x^n=A(x)^k$; after the exponential tilt $e^{\gamma V_k}$, the analogous generating function is $Z_\gamma(x)^k$.  Summing over all possible support sizes therefore gives the geometric series $(1-A(x))^{-1}$ or $(1-Z_\gamma(x))^{-1}$, so the optimal canonical points are the corresponding poles~\cite{BenderEnumeration,PemantleWilsonMelczer}.

\subsection*{Structure of the paper}
\Cref{sec:setup} formulates the general weighted-allocation problem and states the main results, both for the one-statistic model and for an additional profile constraint.  \Cref{sec:examples} develops the principal examples, including the uniform factorial-profile model, the Conway--Maxwell--Poisson and negative-binomial families, and power profiles.  \Cref{sec:lemmas} collects the main estimates and technical tools used throughout, including the canonical variational identities, entropy comparison bounds, and the required one- and two-dimensional local limit estimates, in particular the mixed lattice--nonlattice local central limit theorem.  Finally, \Cref{sec:proofs} combines these ingredients to prove the main canonical, critical, and supercritical results.

\section{Set-up and main results}
\label{sec:setup}

\subsection{One sufficient statistic}
\label{sec:one-stat}

Let $(a_j)_{j\ge1}$ be positive weights and write $A(x)=\sum_{j\ge1}a_jx^j$, which may be infinite.  For $k$ positive parts with total $n$, define the one-statistic microcanonical partition function
\begin{equation}
   W^{(a)}_{n,k}
   =\sum_{\substack{j_1+\cdots+j_k=n\\j_i\ge1}}
      \prod_{i=1}^k a_{j_i}.
   \label{eq:one-W}
\end{equation}
The uniform case corresponds to $a_j=1/j!$.  A multiplicity vector $(j_1,\ldots,j_k)$ then corresponds to $n!\prod_i a_{j_i}$ words.  Thus summing the weights in \eqref{eq:one-W} differs from counting words only by the common factor $n!$, and the maximizing support is the same.  The weighted formulation simply replaces $1/j!$ by general weights $a_j$.

For $x>0$ such that $A(x)<\infty$, let $q_x(j)=a_jx^j/A(x)$.  For fixed $k$, the canonical ensemble is the product measure $q_x^{\otimes k}$, and the exact conditioning identity is
\begin{equation}
   W^{(a)}_{n,k}
   =x^{-n}A(x)^k\Pp_x(T_k=n).
   \label{eq:one-identity}
\end{equation}

Each one-parameter family is considered on its full domain of finite normalization.  It is called \emph{regular} when its natural parameter domain is open~\cite[Section~3]{EckGeyer}.  Here the natural parameter is $\log x$, so this is equivalent to openness of $\{x>0:A(x)<\infty\}$.

\begin{assumption}
\label{ass:one-stat}
The weights $a_j$ are positive, and the one-parameter family $(q_x)_x$ is nonempty and regular.
\end{assumption}

For positive coefficients, this is equivalent to requiring a positive radius of convergence and divergence of the series at that radius whenever it is finite.  Lemma~\ref{lem:one-partition-geometry} gives the corresponding properties of $A$ and of the canonical family.

\begin{theorem}
\label{thm:one-stat}
Under Assumption~\ref{ass:one-stat}, there is a unique $x_*>0$ with $A(x_*)=1$.  The maximizing support fraction is unique and equals $\alpha_*=1/\E_{x_*}J$.  If $k_n$ maximizes $W^{(a)}_{n,k}$, then
\begin{equation}
   \left|\frac{k_n}{n}-\alpha_*\right|=O(n^{-1/2}).
   \label{eq:one-alpha-rate}
\end{equation}
The implicit constant depends only on the one-site law at $x_*$.
\end{theorem}

\subsection{Profile constraints}
\label{sec:two-stat}

The first constraint is the length $T_k=\sum_iJ_i$.  The second is an additive quantity $V_k=\sum_{i=1}^k v(J_i)$ built from the multiplicity profile; for short, we call $V_k$ the \emph{profile}.  We constrain $V_k$ near $\rho n$ and call $\rho=V_k/n$ the \emph{profile rate}.  The normalization is by the fixed length $n$, while the support size $k$ is optimized.  If $k/n\to\alpha$, then the corresponding one-site average is $V_k/k\to\rho/\alpha$.  Throughout the profile-constrained problem we assume $\rho>0$ and use the fixed window
\begin{equation}
  W^{(a,v)}_{n,k}(\rho)
  =\sum_{\substack{j_1+\cdots+j_k=n\\
     |\sum_i v(j_i)-\rho n|\le 1}}
     \prod_{i=1}^k a_{j_i}.
  \label{eq:two-W}
\end{equation}
Whenever $W^{(a,v)}_{n,k}(\rho)>0$, let $\mu^\rho_{n,k}$ be the exchangeable microcanonical law with weights proportional to $\prod_{i=1}^k a_{j_i}$ on the constraint set in \eqref{eq:two-W}, and let $k_n(\rho)$ be any maximizing support.  For $1\le r\le k$, let $\mu^{\rho,[r]}_{n,k}$ denote the joint law under $\mu^\rho_{n,k}$ of the multiplicities of its first $r$ symbols.  For the uniform model with $a_j=1/j!$ and $v(j)=\log(j!)$, this is exactly the statistic $V_k$ introduced in the opening example; equivalently, $\log(n!)-V_k$ is the logarithm of the number of rearrangements.  A profile in this constraint set contributes $n!\prod_i a_{j_i}$ words; the common factor $n!$ does not affect the optimizing support.

The canonical two-parameter one-site family is $q_{x,\gamma}(j)=a_jx^je^{\gamma v(j)}/Z_\gamma(x)$, where $Z_\gamma(x)=\sum_{j\ge1}a_jx^je^{\gamma v(j)}$.  For fixed $k$, the canonical ensemble is $q_{x,\gamma}^{\otimes k}$, while \eqref{eq:two-W} is its microcanonical counterpart.

\begin{assumption}
\label{ass:profile}
The function $v$ is increasing and satisfies
\[
   \frac{v(j)}j\longrightarrow\infty,
   \qquad
   \frac{\Delta v(j)}{v(j)}\longrightarrow0,
   \qquad
   \Delta^2v(j)\longrightarrow0.
\]
The limit $\gamma_c=\lim_{j\to\infty}[-\log a_j]/v(j)$ exists and is finite.  With $\gamma=\gamma_c$ fixed, the one-parameter family $(q_{x,\gamma_c})_x$ is nonempty and regular.
\end{assumption}

Here $\Delta v(j)=v(j+1)-v(j)$ and $\Delta^2v(j)=\Delta v(j+1)-\Delta v(j)$.  Regularity in Assumption~\ref{ass:profile} concerns variation in $x$ with $\gamma_c$ fixed, not the full two-parameter family.  Lemma~\ref{lem:partition-geometry} gives the existence of $x_\gamma$ for every $\gamma\le\gamma_c$.  The fixed-window local limit estimate is given in Lemma~\ref{lem:two-lclt}.

\begin{remark}
The second-difference condition is a sufficient criterion for the nonlattice property used in Lemma~\ref{lem:two-lclt}.  It can be replaced by that property directly.  Profiles with genuine lattice structure, such as $v(j)=j^2$, can also be treated by the corresponding lattice local limit theorem. The present results are restricted to the nonlattice-profile case.
\end{remark}

For every $\gamma\le\gamma_c$, let $x_\gamma$ be the unique solution of $Z_\gamma(x_\gamma)=1$ and let $q_\gamma(j)=a_jx_\gamma^je^{\gamma v(j)}$.  Set
\begin{equation}
  \alpha_\gamma=\frac1{\E_\gamma J},
  \qquad
  \rho_\gamma=\frac{\E_\gamma v(J)}{\E_\gamma J}.
  \label{eq:alpharho}
\end{equation}
At the boundary we abbreviate $x_c=x_{\gamma_c}$, $\alpha_c=\alpha_{\gamma_c}$, $\rho_c=\rho_{\gamma_c}$, and $q_c=q_{\gamma_c}$.  Lemma~\ref{lem:exp-family} shows that $\rho_\gamma$ is strictly increasing for $\gamma<\gamma_c$, while Lemma~\ref{lem:partition-geometry} gives $(\alpha_\gamma,\rho_\gamma)\to(\alpha_c,\rho_c)$ as $\gamma\uparrow\gamma_c$.

Throughout this subsection the model $(a,v)$ is fixed.  Constants may depend on it; their subscripts record any additional dependence.  We use $\dTV(\mu,\pi)=\frac12\sum_z|\mu(z)-\pi(z)|$ for total variation distance.

\begin{theorem}
\label{thm:two-canonical}
Under Assumption~\ref{ass:profile}, let $\gamma\le\gamma_c$ be such that $\rho=\rho_\gamma>0$.
\begin{enumerate}[label=(\roman*)]
\item If $\gamma<\gamma_c$, then
\[
 \left|\frac{k_n(\rho)}{n}-\alpha_\gamma\right|=O_\gamma(n^{-1/2}).
\]
\item At $\gamma=\gamma_c$,
\[
 \left|\frac{k_n(\rho_c)}{n}-\alpha_c\right|
 =O\!\left(\left(\frac{\log n}{n}\right)^{1/2}\right).
\]
\item For every $\gamma\le\gamma_c$ and every fixed $r$,
\[
 \dTV\!\left(\mu^{\rho,[r]}_{n,k_n(\rho)},q_\gamma^{\otimes r}\right)
 =O_{\gamma,r}\!\left(\left(\frac{\log n}{n}\right)^{1/2}\right).
\]
\end{enumerate}
\end{theorem}

For $\rho>\rho_c$, Assumption~\ref{ass:profile} implies that a single symbol of sublinear multiplicity can carry a linear excess in $V_k$.  Let $s_n=s_n(\rho)$ be the smallest $m$ such that
$v(m)-\rho_cm\ge(\rho-\rho_c)n$, and let $\xi_n$ be the overshoot,
$\xi_n=v(s_n)-\rho_cs_n-(\rho-\rho_c)n$.  Then $s_n=o(n)$, $\xi_n=o(n)$, and $-\log q_c(s_n)=o(n)$; these consequences of Assumption~\ref{ass:profile} are checked in the proof of Theorem~\ref{thm:two-supercritical}.  Set
\[
 \omega_n(\rho)=\left(\frac{-\log q_c(s_n)+\log n}{n}\right)^{1/2}+\frac{\xi_n}{n},
\]
so that $\omega_n(\rho)\to0$; for fixed $\rho$ we write simply $\omega_n$.

\begin{theorem}
\label{thm:two-supercritical}
Under Assumption~\ref{ass:profile}, fix $\rho>\rho_c$ and let $k=k_n(\rho)$.
\begin{enumerate}[label=(\roman*)]
\item The optimal fraction is pinned at the boundary value:
\[
 \alpha^*(\rho)=\alpha_c,
 \qquad
 \left|\frac{k_n(\rho)}n-\alpha_c\right|=O_\rho(\omega_n).
\]
\item For every fixed $r$,
\[
 \dTV\!\left(\mu^{\rho,[r]}_{n,k_n(\rho)},q_c^{\otimes r}\right)=O_{\rho,r}(\omega_n).
\]
\item There exists a deterministic sequence $M_n\to\infty$ such that, with
$ I_n=\{i:J_i>M_n\}$,
\[
 \E\frac{|I_n|}{k}=o(1),
 \qquad
 \E\left|\frac1n\sum_{i\in I_n}v(J_i)-(\rho-\rho_c)\right|=o(1).
\]
\end{enumerate}
\end{theorem}

Thus the excess profile is carried by a vanishing fraction of symbols.

\begin{remark}
The pinning transition has a kink in the optimal support fraction whenever
\[
 \Cov_c\!\bigl(J,v(J)-\rho_cJ\bigr)\ne0.
\]
Indeed, along the canonical branch,
\[
 \frac{\dd\alpha^*}{\dd\rho}
 =
 -\alpha_\gamma\,
 \frac{\Cov_\gamma(J,v(J)-\rho_\gamma J)}
      {\Var_\gamma(v(J)-\rho_\gamma J)}.
\]
Thus the left derivative at $\rho_c$ is given by the same expression under the boundary law $q_c$, while Theorem~\ref{thm:two-supercritical} gives $\alpha^*(\rho)=\alpha_c$ for $\rho>\rho_c$, so the right derivative is zero.  Hence a nonzero boundary covariance gives a discontinuity of the first derivative.

For the factorial profile $v(j)=\log(j!)$, the boundary law is $q_c(j)=2^{-j}$ and the covariance is positive.  In this case
\[
 (\alpha^*)'(\rho^*-)=-0.554771\ldots,
 \qquad
 (\alpha^*)'(\rho^*+)=0.
\]
\end{remark}

\begin{remark}
If instead $-\log a_j/v(j)\to\infty$, while $v$ satisfies the first two conditions in Assumption~\ref{ass:profile}, then $Z_\gamma$ is entire for every finite $\gamma$.  The unit-partition curve therefore has no finite canonical endpoint; Lemma~\ref{lem:no-finite-boundary} shows that $\rho_\gamma\to\infty$, so the canonical branch extends to arbitrarily large profile rates and the pinning mechanism of \Cref{thm:two-supercritical} has no finite transition point.
\label{rem:infinite-gamma}
\end{remark}

\section{Examples}
\label{sec:examples}

The examples begin with the one-statistic negative-binomial family and its positive-Poisson limit, followed by models with profile constraints.

\subsection{One sufficient statistic}

\paragraph{Negative-binomial weights.}
Take $a_j^{(\theta)}=(\theta)_j/j!$ with $\theta>0$.

\begin{corollary}
\label{cor:nb-one-stat}
The canonical law is the zero-truncated negative-binomial distribution, and the unique limiting optimal fraction is
\[
 \alpha_\theta^*=\frac{1}{2\theta(2^{1/\theta}-1)}.
\]
If $k_n$ maximizes the one-statistic partition function, then $|k_n/n-\alpha_\theta^*|=O_\theta(n^{-1/2})$.
The corresponding microcanonical law is the symmetric Dirichlet--multinomial law conditioned on full occupancy.
\end{corollary}

\begin{proof}
We first determine the optimal fraction.  Here $A_\theta(x)=(1-x)^{-\theta}-1$ is finite exactly for $0<x<1$, so the family is nonempty and regular.  The equation $A_\theta(x_*)=1$ gives $x_*=1-2^{-1/\theta}$.  Differentiating $A_\theta$ at $x_*$ gives $\E_{x_*}J=2\theta(2^{1/\theta}-1)$.  Theorem~\ref{thm:one-stat} therefore gives the formula for $\alpha_\theta^*$ and the $O(n^{-1/2})$ rate.

To identify the microcanonical law, condition the canonical product law on $J_1+\cdots+J_k=n$.  This removes the common factor $x^n$ and leaves weights proportional to $\prod_i(\theta)_{j_i}/j_i!$.  These are exactly the symmetric Dirichlet--multinomial weights, conditioned on every symbol appearing.
\end{proof}

\begin{remark}
The positive-Poisson model is recovered as $\theta\to\infty$.  If $x=\lambda/\theta$, the zero-truncated negative-binomial law converges to the positive-Poisson law with parameter $\lambda$, while $\theta x_*\to\log2$ and $\alpha_\theta^*\to1/(2\log2)$.  At the microcanonical level, $(\theta)_j\sim\theta^j$ for each fixed $j$.  After conditioning on $\sum_iJ_i=n$, the common factor $\theta^n$ disappears and the weights converge to $\prod_i1/j_i!$.  Thus the uniform allocation model from the introduction is the Poisson limit of this one-statistic family.
\end{remark}

\subsection{Profile constraints}

\paragraph{The factorial profile and the Conway--Maxwell--Poisson family.}
Take $a_j=1/j!$ and $v(j)=\log(j!)$.  Let $\alpha^*(\rho)$ denote the limiting optimal support fraction, and put
\[
 q^*(j)=2^{-j},\qquad j\ge1,
 \qquad
 \rho^*=0.507834\ldots.
\]

\begin{corollary}
\label{cor:typeclass}
For the uniform model with the factorial profile constraint, the critical canonical law is $q^*$, the critical profile rate is $\rho^*$, and the critical support fraction is $1/2$. One has $\alpha^*(0)=1$, and $\alpha^*$ decreases strictly on $[0,\rho^*]$. Then:
\begin{enumerate}[label=(\roman*)]
\item for each $\rho\in(0,\rho^*)$ there is a unique $\gamma<1$ with $\rho_\gamma=\rho$, and
\[
 \alpha^*(\rho)=\alpha_\gamma.
\]
Moreover, for every fixed $r$,
\[
 \left|\frac{k_n(\rho)}n-\alpha_\gamma\right|=O_\rho(n^{-1/2}),
 \qquad
 \dTV(\mu^{\rho,[r]}_{n,k_n(\rho)},q_\gamma^{\otimes r})
 =O_{\rho,r}\!\left((\log n/n)^{1/2}\right);
\]
\item at $\rho=\rho^*$ one has $\alpha^*(\rho^*)=1/2$ and $q_c=q^*$, and both the support error and every fixed-$r$ total-variation error are $O_r((\log n/n)^{1/2})$;
\item for every $\rho>\rho^*$ one has $\alpha^*(\rho)=1/2$, and
\[
 \left|\frac{k_n(\rho)}n-\frac12\right|
 +\dTV(\mu^{\rho,[r]}_{n,k_n(\rho)},(q^*)^{\otimes r})
 =O_{\rho,r}((\log n)^{-1/2}),
\]
while a vanishing fraction of symbols carries the excess $(\rho-\rho^*)n$ in $V_k$.
\end{enumerate}
\end{corollary}

\begin{proof}
For every $j\ge2$, $(-\log a_j)/v(j)=1$, so the canonical endpoint is $\gamma_c=1$.  At this boundary,
\[
 Z_1(x)=\sum_{j\ge1}x^j=\frac{x}{1-x},
\]
and the unit-partition equation $Z_1(x_c)=1$ gives $x_c=1/2$.  Therefore
\[
 q_c(j)=x_c^j=2^{-j}=q^*(j),
 \qquad
 \E_cJ=2,
 \qquad
 \alpha_c=\frac1{\E_cJ}=\frac12.
\]
Moreover,
\[
 \rho_c=\frac{\E_c\log(J!)}{\E_cJ}
 =\frac12\E_c\log(J!)
 =\rho^*,
\]
where the last quantity is the constant computed in the introduction.

As $\gamma\to-\infty$ along the unit-partition curve, the canonical law concentrates at $1$.  Hence $\rho_\gamma\downarrow0$ and $\alpha_\gamma\uparrow1$.  Together with the strict monotonicity of $\rho_\gamma$ from Lemma~\ref{lem:exp-family}, this identifies the canonical branch as $\rho\in(0,\rho^*)$.  Theorem~\ref{thm:two-canonical} gives (i)--(ii).

To verify strict decrease, let $\widehat q_\gamma(j)=\alpha_\gamma j q_\gamma(j)$ be the size-biased law. Differentiation along the unit-partition curve gives
\[
 \frac{d\alpha_\gamma}{d\gamma}
 =-\alpha_\gamma^2\Cov_\gamma\!\bigl(J,v(J)-\rho_\gamma J\bigr)
 =-\alpha_\gamma\Cov_{\widehat q_\gamma}\!\left(J,\frac{v(J)}J\right)<0,
\]
since $v(j)/j=\log(j!)/j$ is strictly increasing. Together with $d\rho_\gamma/d\gamma>0$ and continuity at the endpoints, this proves the claimed monotonicity. At $\rho=0$, the constraint gives $(\log2)(n-k)\le V_k\le1$, so $\alpha^*(0)=1$.

For (iii), the exceptional multiplicity $s_n$ is determined by $\log(s_n!)-\rho^*s_n\asymp_\rho n$.  Since $\log(s!)\sim s\log s$, this gives $s_n\asymp_\rho n/\log n$.  The overshoot is bounded by one increment of $\log(s!) - \rho^*s$, so $\xi_n=O_\rho(\log n)$.  Finally, under $q^*$, $-\log q^*(s_n)=(\log2)s_n=O_\rho(n/\log n)$.  Substituting these estimates into $\omega_n(\rho)$ gives $\omega_n(\rho)=O_\rho((\log n)^{-1/2})$.  Theorem~\ref{thm:two-supercritical} now gives the pinned support, total-variation, and localization statements.
\end{proof}

\paragraph{The COM-negative-binomial family.}
Fix $\theta>0$, take $a_j=(\theta)_j/j!$, and again let $v(j)=\log(j!)$.  The canonical law is
$q_{\theta,x,\gamma}(j)=(\theta)_jx^j(j!)^{\gamma-1}/Z_{\theta,\gamma}(x)$, the zero-truncated COM-negative-binomial family of Chakraborty--Ong~\cite{ChakrabortyOng}.

\begin{corollary}
\label{cor:comnb}
Fix $\theta>0$, and let $\alpha_\theta^*(\rho)$ denote the limiting optimal support fraction for the COM-negative-binomial family.  Let
\[
 x_{\theta,c}=1-2^{-1/\theta},
 \qquad
 q_{\theta,c}(j)=\frac{(\theta)_j}{j!}\bigl(1-2^{-1/\theta}\bigr)^j,
 \qquad j\ge1.
\]
If $J\sim q_{\theta,c}$, define
\[
 \alpha_{\theta,c}=\frac1{\E J},
 \qquad
 \rho_{\theta,c}=\frac{\E\log(J!)}{\E J}.
\]
Then
\[
 \alpha_{\theta,c}=\frac1{2\theta(2^{1/\theta}-1)},
 \qquad
 \rho_{\theta,c}=\alpha_{\theta,c}\sum_{j\ge1}
 \frac{(\theta)_j}{j!}\bigl(1-2^{-1/\theta}\bigr)^j\log(j!).
\]
In particular,
\[
 \alpha_{\theta,c}\sim \frac{2^{-1/\theta}}{2\theta}\longrightarrow0
 \quad(\theta\downarrow0),
 \qquad
 \alpha_{\theta,c}\longrightarrow\frac1{2\log2}
 \quad(\theta\to\infty).
\]
Thus the finite-$\theta$ critical fraction ranges from $0$ to $1/(2\log2)$ as $\theta$ varies, although the limiting $\theta=\infty$ COM--Poisson model has a different critical fraction, namely $1/2$; see the remark below.
Then:
\begin{enumerate}[label=(\roman*)]
\item for each $\rho\in(0,\rho_{\theta,c})$ there is a unique $\gamma<0$ with $\rho_\gamma=\rho$, and
\[
 \alpha_\theta^*(\rho)=\alpha_\gamma;
\]
moreover, for every fixed $r$,
\[
 \left|\frac{k_n(\rho)}n-\alpha_\gamma\right|=O_\rho(n^{-1/2}),
 \qquad
 \dTV(\mu^{\rho,[r]}_{n,k_n(\rho)},q_\gamma^{\otimes r})=O_{\rho,r}\!\left((\log n/n)^{1/2}\right);
\]
\item at $\rho=\rho_{\theta,c}$ one has $\alpha_\theta^*(\rho_{\theta,c})=\alpha_{\theta,c}$, and both the support error and every fixed-$r$ total-variation error are $O_r((\log n/n)^{1/2})$;
\item for every $\rho>\rho_{\theta,c}$ one has $\alpha_\theta^*(\rho)=\alpha_{\theta,c}$, and
\[
 \left|\frac{k_n(\rho)}n-\alpha_{\theta,c}\right|+
 \dTV(\mu^{\rho,[r]}_{n,k_n(\rho)},q_{\theta,c}^{\otimes r})
 =O_{\theta,\rho,r}((\log n)^{-1/2}),
\]
while a vanishing fraction of symbols carries the excess $(\rho-\rho_{\theta,c})n$ in $V_k$.
\end{enumerate}
\end{corollary}

\begin{proof}
Since $a_j=(\theta)_j/j!$, one has $(-\log a_j)/\log(j!)\to0$, so again $\gamma_c=0$.  At the boundary, $Z_{\theta,0}(x)=\sum_{j\ge1}(\theta)_jx^j/j!=(1-x)^{-\theta}-1$, which is finite exactly for $0<x<1$.  The unit-partition equation is therefore $(1-x)^{-\theta}=2$, so
\[
 x_{\theta,c}=1-2^{-1/\theta},
 \qquad
 q_{\theta,c}(j)=\frac{(\theta)_j}{j!}\bigl(1-2^{-1/\theta}\bigr)^j.
\]
Differentiating $Z_{\theta,0}(x)$ gives
\[
 \E J=x\frac{\dd}{\dd x}Z_{\theta,0}(x)\bigg|_{x=x_{\theta,c}}
 =x_{\theta,c}\theta(1-x_{\theta,c})^{-\theta-1}
 =2\theta(2^{1/\theta}-1),
\]
which yields
\[
 \alpha_{\theta,c}=\frac1{\E J}=\frac1{2\theta(2^{1/\theta}-1)}.
\]
This formula already shows the two extreme finite-$\theta$ regimes.  As $\theta\downarrow0$,
\[
 \E J\sim 2\theta 2^{1/\theta}\longrightarrow\infty,
 \qquad
 \alpha_{\theta,c}\sim\frac{2^{-1/\theta}}{2\theta}\longrightarrow0.
\]
Thus the boundary bulk uses larger and larger multiplicities and only a vanishing fraction of symbols.  On the other hand, $2^{1/\theta}-1=(\log2)/\theta+O(\theta^{-2})$, so $\alpha_{\theta,c}\to1/(2\log2)$ as $\theta\to\infty$.
The stated expression for $\rho_{\theta,c}$ is simply the definition of $\rho_c$ in \eqref{eq:alpharho}.

As $\gamma\to-\infty$ along the unit-partition curve, the law $q_\gamma$ converges to the point mass at $1$, so $\rho_\gamma\downarrow0$ and $\alpha_\gamma\uparrow1$.  Together with the strict monotonicity from Lemma~\ref{lem:exp-family}, this identifies the canonical branch as $\rho\in(0,\rho_{\theta,c})$.  Theorem~\ref{thm:two-canonical} gives (i)--(ii).

For (iii), the boundary law has a geometric tail up to a polynomial factor because $(\theta)_j/j!\asymp_\theta j^{\theta-1}$.  Hence $-\log q_{\theta,c}(j)=c_\theta j+O_\theta(\log j)$ for some $c_\theta>0$.  As in the factorial-profile example, this implies $s_n\asymp_\rho n/\log n$, $\xi_n=O_\rho(\log n)$, and therefore $\omega_n(\rho)=O_{\theta,\rho}((\log n)^{-1/2})$. The rest then follows from Theorem~\ref{thm:two-supercritical}.
\end{proof}

\begin{figure}[H]
\centering
\includegraphics[width=0.72\textwidth]{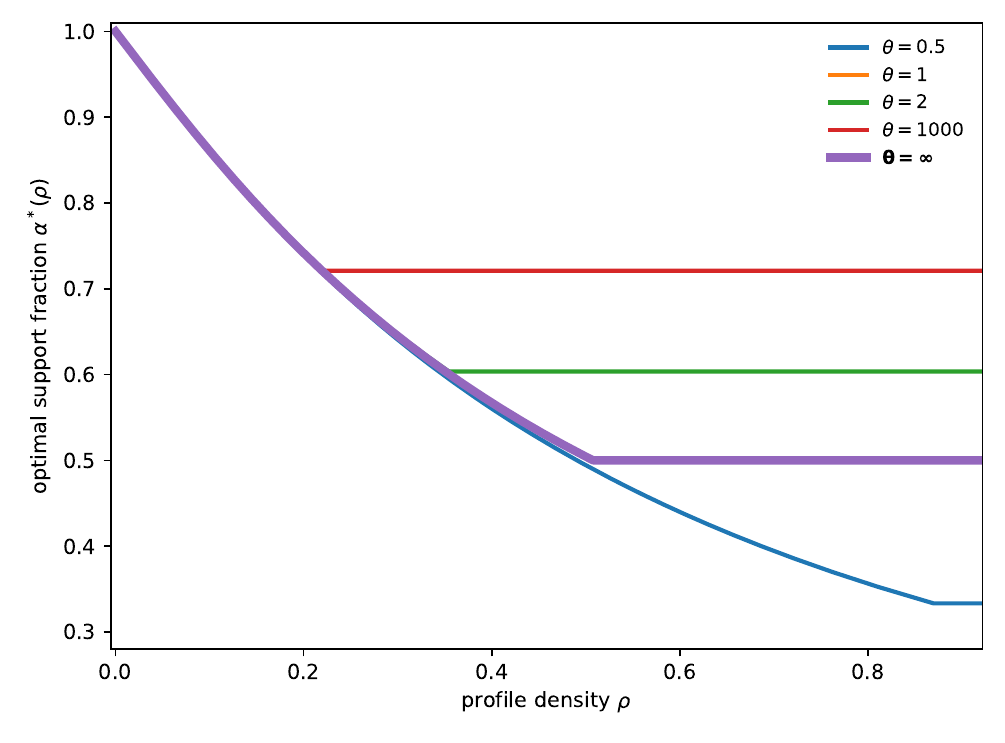}
\caption{The limiting optimal curves for the COM-negative-binomial family.  The bold curve corresponds to the limiting COM--Poisson model $(\theta=\infty)$.}
\label{fig:comnb-curves}
\end{figure}

If $k\sim\alpha n$ bulk multiplicities have law $J$, then $k\E J\sim n$ gives $\alpha=1/\E J$, while $k\E\log(J!)\sim\rho n$ gives $\rho=\E\log(J!)/\E J$.  Thus $(\alpha_{\theta,c},\rho_{\theta,c})$ is exactly the support fraction and the value of $V_k/n$ produced by the boundary zero-truncated negative-binomial bulk.

\begin{remark}
At $\theta=1$, since $(1)_j=j!$,
\[
 q_{1,x,\gamma}(j)\propto x^j(j!)^\gamma
 =q_{\infty,x,\gamma+1}(j),
\]
so the $\theta=1$ and $\theta=\infty$ optimal curves coincide after the reparametrization $\widetilde\gamma=\gamma+1$.

The limit $\theta\to\infty$ is nevertheless singular at the canonical endpoint.  Under the scaling $x=\lambda/\theta$, one has $(\theta)_j/\theta^j\to1$ for each fixed $j$, but the endpoint is determined by the large-$j$ tail.  For every finite $\theta$, $(\theta)_j/j!\asymp_\theta j^{\theta-1}$ and hence $\gamma_c(\theta)=0$, whereas the limiting COM--Poisson model has $a_j=1/j!$ and $\gamma_c(\infty)=1$.  Consequently the finite-$\theta$ critical points converge to the positive-Poisson point $(\rho,\alpha)=(0.219341\ldots,1/(2\log2))$, which is an interior point of the limiting COM--Poisson branch; that branch pins later at $(\rho_c,\alpha_c)=(0.507834\ldots,1/2)$.
\end{remark}

\paragraph{Power profiles.}
Take $a_j=1/j!$ and $v(j)=j^p$ with $1<p<2$.  Then
$\Delta^2v(j)\sim p(p-1)j^{p-2}$, so the nonlattice condition in Assumption~\ref{ass:profile} holds.

\begin{corollary}
\label{cor:power-profile}
Fix $1<p<2$, and let $\alpha_p^*(\rho)$ denote the limiting optimal support fraction for the profile $v(j)=j^p$.  Let $q_c$ be the positive-Poisson law with parameter $\log2$, so
\[
 q_c(j)=\frac{(\log2)^j}{j!},\qquad j\ge1,
\]
and define
\[
 \alpha_c=\frac1{2\log2},
 \qquad
 \rho_c=\frac{\E_cJ^p}{\E_cJ}=\frac1{2\log2}\E_cJ^p.
\]
At the endpoint $\rho=1$, the constraint implies $n-k=O(1)$, so $\alpha_p^*(1)=1$.  For $\rho>1$:
\begin{enumerate}[label=(\roman*)]
\item for each $\rho\in(1,\rho_c)$ there is a unique $\gamma<0$ such that $\rho_\gamma=\rho$, and
\[
 \alpha_p^*(\rho)=\alpha_\gamma.
\]
Moreover, if $k_n(\rho)$ is an optimizing support, then
\[
 \left|\frac{k_n(\rho)}n-\alpha_\gamma\right|=O_\rho(n^{-1/2}),
 \qquad
 \dTV(\mu^{\rho,[r]}_{n,k_n(\rho)},q_\gamma^{\otimes r})=O_{\rho,r}\!\left((\log n/n)^{1/2}\right)
\]
for every fixed $r$.
\item at $\rho=\rho_c$, one has $\alpha_p^*(\rho_c)=\alpha_c$, and both the support error and every fixed-$r$ total-variation error are $O_r((\log n/n)^{1/2})$.
\item for every $\rho>\rho_c$, one has $\alpha_p^*(\rho)=\alpha_c$.  Moreover,
\[
 \left|\frac{k_n(\rho)}n-\alpha_c\right|+
 \dTV(\mu^{\rho,[r]}_{n,k_n(\rho)},q_c^{\otimes r})
 =O_{\rho,r}\!\left(n^{-(p-1)/(2p)}(\log n)^{1/2}+n^{-1/p}\right),
\]
again for every fixed $r$, and the excess $(\rho-\rho_c)n$ in $V_k$ is carried by a vanishing fraction of symbols.
\end{enumerate}
\end{corollary}

\begin{proof}
At $\rho=1$, convexity gives $(p-1)(n-k)\le V_k-n\le1$, proving the endpoint assertion.
For $\rho>1$, since $a_j=1/j!$, one has $(-\log a_j)/v(j)=\log(j!)/j^p\to0$, so $\gamma_c=0$.  At the boundary $\gamma=0$, $Z_0(x)=\sum_{j\ge1}x^j/j!=e^x-1$, which is finite for every $x>0$.  The unit-partition equation $Z_0(x_c)=1$ therefore becomes $e^{x_c}-1=1$, so
\[
 x_c=\log2,
 \qquad
 q_c(j)=\frac{(\log2)^j}{j!},\quad j\ge1.
\]
Under this law, $\E_cJ=\sum_{j\ge1}j(\log2)^j/j!=2\log2$, whence
\[
 \alpha_c=\frac1{\E_cJ}=\frac1{2\log2},
 \qquad
 \rho_c=\frac{\E_cJ^p}{\E_cJ}=\frac1{2\log2}\E_cJ^p.
\]

As $\gamma\to-\infty$ along the unit-partition curve, the law $q_\gamma$ concentrates at $1$.  Hence $\rho_\gamma\downarrow1$ and $\alpha_\gamma\uparrow1$.  Together with the strict monotonicity of $\rho_\gamma$ from Lemma~\ref{lem:exp-family}, this shows that the canonical branch is exactly $\rho\in(1,\rho_c)$.  Theorem~\ref{thm:two-canonical} now gives (i)--(ii).

For $\rho>\rho_c$, let $s_n$ be the smallest integer such that $s_n^p-\rho_cs_n\ge(\rho-\rho_c)n$.  Since $p>1$, this implies $s_n\asymp_\rho n^{1/p}$.  The overshoot $\xi_n=s_n^p-\rho_cs_n-(\rho-\rho_c)n$ is controlled by one increment of $m\mapsto m^p-\rho_cm$, namely
\[
 0\le \xi_n\le s_n^p-(s_n-1)^p+\rho_c=O_\rho(s_n^{p-1})=O_\rho(n^{(p-1)/p}).
\]
Moreover, since $q_c(j)=(\log2)^j/j!$, Stirling's formula gives
\[
 -\log q_c(s_n)=\log(s_n!)-s_n\log\log2=O_\rho(s_n\log s_n)=O_\rho(n^{1/p}\log n).
\]
Substituting these estimates into $\omega_n(\rho)$ yields
\[
 \omega_n(\rho)=O_\rho\!\left(n^{-(p-1)/(2p)}(\log n)^{1/2}+n^{-1/p}\right).
\]
Theorem~\ref{thm:two-supercritical} now gives (iii).
\end{proof}

\begin{figure}[H]
\centering
\includegraphics[width=0.72\textwidth]{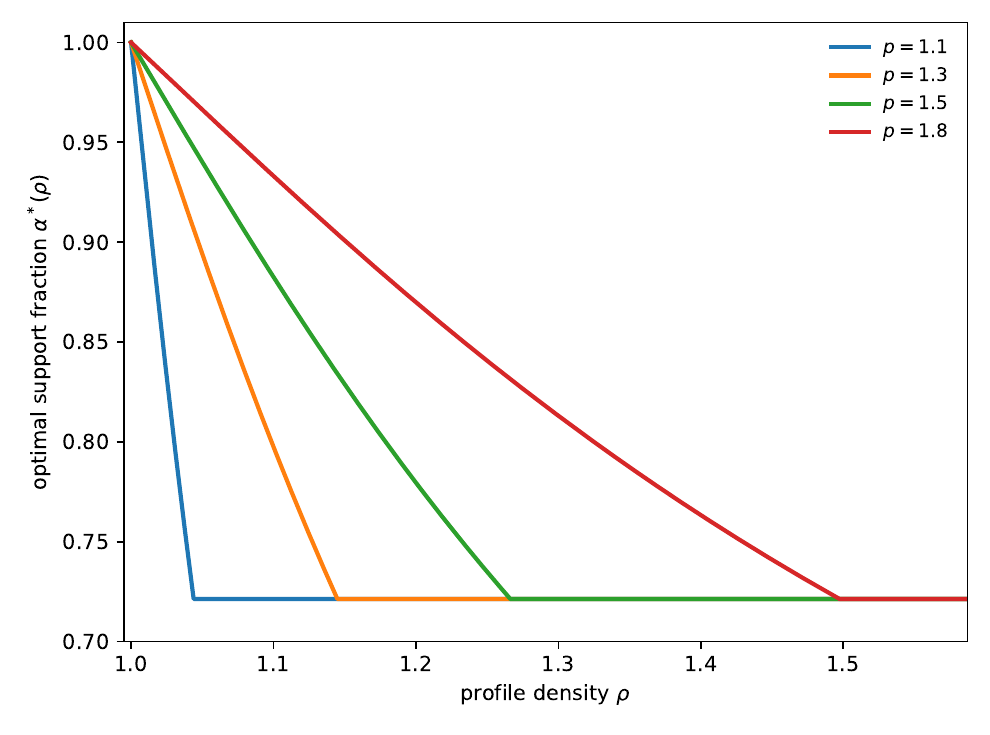}
\caption{The limiting optimal curves for the power profiles $v(j)=j^p$. As $\rho$ increases, each curve reaches the same pinned fraction $\alpha_c=1/(2\log2)$, at a threshold depending on $p$.}
\label{fig:power-curves}
\end{figure}

\paragraph{An example with no finite transition.}
Take $a_j=1/j!$ and $v(j)=j\sqrt{\log(j+1)}$.

\begin{corollary}
\label{cor:no-finite-transition}
Every finite profile rate on the canonical branch is attained at a finite canonical parameter.  In particular, this model has no finite pinning transition of the type in Theorem~\ref{thm:two-supercritical}.
\end{corollary}

\begin{proof}
Since $\log(j!)/(j\sqrt{\log(j+1)})\to\infty$, one has $-\log a_j/v(j)\to\infty$.  Remark~\ref{rem:infinite-gamma}, together with Lemma~\ref{lem:no-finite-boundary}, gives $\rho_\gamma\to\infty$ and proves the claim.
\end{proof}

\section{Main lemmas}
\label{sec:lemmas}

This section isolates the two inputs used repeatedly in the proofs: the geometry of the canonical normalizing functions and the local-limit estimates needed to compare canonical and microcanonical measures.

\subsection{Properties of the canonical normalizing functions}

The coefficient assumptions determine where the canonical normalizing functions are finite.

\begin{lemma}
\label{lem:one-partition-geometry}
Under Assumption~\ref{ass:one-stat}, the equation $A(x)=1$ has a unique solution $x_*>0$, and $A$ is finite in a neighborhood of $x_*$.  The canonical mean $xA'(x)/A(x)$ is strictly increasing throughout the domain of finite normalization.
\end{lemma}

\begin{proof}
We first prove the existence and uniqueness of $x_*$.  On its domain of finiteness, $A$ is continuous and strictly increasing, and $A(x)\downarrow0$ as $x\downarrow0$.  If the domain is unbounded, $A(x)\ge a_1x$ gives $A(x)\to\infty$.  If its right endpoint is finite, regularity excludes that endpoint from the full domain.  The series therefore diverges there, and monotone convergence gives divergence as the endpoint is approached from below.  Hence $A$ crosses the level $1$ exactly once.  Regularity also gives finiteness in a neighborhood of the solution.

To prove strict monotonicity of the mean, differentiate under the canonical law $q_x(j)=a_jx^j/A(x)$:
\[
   \frac{\dd}{\dd\log x}\E_xJ=\Var_x(J)>0.
\]
Thus the canonical mean is strictly increasing.
\end{proof}

\begin{lemma}
\label{lem:partition-geometry}
Under Assumption~\ref{ass:profile}, the canonical partition functions have the following domain structure:
\begin{enumerate}[label=(\roman*)]
\item if $\gamma<\gamma_c$, then $Z_\gamma(x)<\infty$ for every $x>0$ and $Z_\gamma(x)\to\infty$ as $x\to\infty$;
\item if $\gamma=\gamma_c$, then $Z_{\gamma_c}(x)=1$ has a unique solution $x_c>0$, and $Z_{\gamma_c}$ is finite in a neighborhood of $x_c$;
\item if $\gamma>\gamma_c$, then $Z_\gamma(x)=\infty$ for every $x>0$.
\end{enumerate}
Consequently, for every $\gamma\le\gamma_c$ the equation $Z_\gamma(x)=1$ has a unique solution $x_\gamma$.  Moreover $x_\gamma\to x_c$ as $\gamma\uparrow\gamma_c$, and the laws $q_\gamma$ have a common exponential moment in $J$ for $\gamma\le\gamma_c$ sufficiently close to $\gamma_c$.
\end{lemma}

\begin{proof}
We start by proving (i).  Fix $\gamma<\gamma_c$ and write $\delta=\gamma_c-\gamma>0$.  By the definition of $\gamma_c$,
\[
  \log a_j+\gamma v(j)=-\delta v(j)+o(v(j)).
\]
For all sufficiently large $j$ the right-hand side is at most $-\delta v(j)/2$.  Since $v(j)/j\to\infty$, for every fixed $x>0$,
\[
  \log\bigl(a_jx^je^{\gamma v(j)}\bigr)
  \le -\frac{\delta}{2}v(j)+j\log x\le -j
\]
for all large $j$.  Thus $Z_\gamma(x)<\infty$ for every $x>0$.  Since $Z_\gamma(x)\ge a_1xe^{\gamma v(1)}$, it also follows that $Z_\gamma(x)\to\infty$ as $x\to\infty$.  This proves (i).

We next prove (ii).  With $\gamma=\gamma_c$ fixed, Assumption~\ref{ass:profile} gives a nonempty regular family in $x$.  Applying Lemma~\ref{lem:one-partition-geometry} to the positive weights $a_je^{\gamma_cv(j)}$ gives a unique solution $x_c$ of $Z_{\gamma_c}(x)=1$ and finiteness in a neighborhood of $x_c$.  This proves (ii).

To prove (iii), fix $\gamma>\gamma_c$ and put $\delta=\gamma-\gamma_c>0$.  Then
\[
  \log a_j+\gamma v(j)=\delta v(j)+o(v(j)).
\]
For every fixed $x>0$, the superlinear growth $v(j)/j\to\infty$ implies
\[
  \log\bigl(a_jx^je^{\gamma v(j)}\bigr)
  \ge \frac{\delta}{2}v(j)+j\log x\longrightarrow\infty.
\]
The summands therefore do not tend to zero, so $Z_\gamma(x)=\infty$.  This proves (iii).

We next establish the existence of $x_\gamma$ and its convergence to $x_c$.  For $\gamma<\gamma_c$, part (i) and positivity of the coefficients show that $Z_\gamma$ is continuous, strictly increasing from zero, and unbounded.  Thus it crosses $1$ exactly once.  Part (ii) gives the same conclusion at $\gamma_c$.

For $x$ in a sufficiently small neighborhood of $x_c$, dominated convergence gives $Z_\gamma(x)\to Z_{\gamma_c}(x)$ as $\gamma\uparrow\gamma_c$.  Indeed, $v$ is eventually nonnegative, so the tail terms for $\gamma\le\gamma_c$ are bounded by those of the finite series $Z_{\gamma_c}(x)$; the finitely many remaining terms are bounded on a fixed parameter interval.  For every sufficiently small $\varepsilon>0$, the inequalities $Z_{\gamma_c}(x_c-\varepsilon)<1<Z_{\gamma_c}(x_c+\varepsilon)$ therefore also hold with $\gamma$ in place of $\gamma_c$ when $\gamma$ is sufficiently close from below.  This places $x_\gamma$ between $x_c-\varepsilon$ and $x_c+\varepsilon$, proving convergence.

It remains to establish a common exponential moment.  By part (ii), choose $t>0$ such that $Z_{\gamma_c}(x_ce^{2t})<\infty$.  The convergence just proved gives $x_\gamma\le x_ce^t$ for $\gamma\le\gamma_c$ sufficiently close to $\gamma_c$.  Hence
\[
   \E_\gamma e^{tJ}
   =Z_\gamma(x_\gamma e^t)
   \le Z_\gamma(x_ce^{2t}).
\]
The same domination argument bounds the right-hand side uniformly for $\gamma\in[\gamma_c-\eta,\gamma_c]$, for some $\eta>0$.  Thus the same exponent $t$ and a finite bound apply to all these laws, including $q_c$.
\end{proof}

The following calculation identifies the optimal canonical curve.  Fix a profile rate $\rho>0$.  For each support fraction $\alpha$ in the interior region under consideration, let $x_\rho(\alpha)$ and $\gamma_\rho(\alpha)$ denote the canonical parameters that match the two constraints:
\[
   \E_{x_\rho(\alpha),\gamma_\rho(\alpha)}J=\frac1\alpha,
   \qquad
   \E_{x_\rho(\alpha),\gamma_\rho(\alpha)}v(J)=\frac\rho\alpha.
\]
Thus, when $k=\alpha n$, the canonical law has the same mean length and mean profile as the microcanonical constraint.  Let
$E_{n,k}(\rho)=\{T_k=n,\ |V_k-\rho n|\le1\}$.  The definition of $q_{x,\gamma}$ gives
\[
 W^{(a,v)}_{n,k}(\rho)
 =Z_\gamma(x)^k x^{-n}e^{-\gamma\rho n}
   \E_{x,\gamma}\!\left[
      e^{-\gamma(V_k-\rho n)}\mathbf1_{E_{n,k}(\rho)}
   \right],
\]
where here and below $(x,\gamma)=(x_\rho(\alpha),\gamma_\rho(\alpha))$.  On $E_{n,k}(\rho)$, the extra exponential factor is bounded above and below by constants depending only on $\gamma$.  Since the two constraints are centered under the matched law, Lemma~\ref{lem:two-lclt} gives $\Pp_{x,\gamma}(E_{n,k}(\rho))=\Theta(n^{-1})$, uniformly when the matched parameters remain in a compact set.  Hence
\[
   \log W^{(a,v)}_{n,k}(\rho)
   =nF_\rho(\alpha)-\log n+O(1),
\]
with
\[
   F_\rho(\alpha)
   =\alpha\log Z_{\gamma_\rho(\alpha)}\!\bigl(x_\rho(\alpha)\bigr)
    -\log x_\rho(\alpha)-\rho\,\gamma_\rho(\alpha).
\]
Thus $F_\rho$ is the leading exponential contribution associated with the support fraction $\alpha$.  The next lemma shows that it is strictly concave and identifies its maximizer.

\begin{lemma}
\label{lem:exp-family}
Fix $\rho$ and let $x_\rho(\alpha),\gamma_\rho(\alpha)$ be the matched parameters above.  If $\Sigma_{\alpha,\rho}$ denotes the covariance matrix of $(J,v(J))$ under the matched canonical law, then
\[
   F_\rho'(\alpha)
   =\log Z_{\gamma_\rho(\alpha)}\!\bigl(x_\rho(\alpha)\bigr),
   \qquad
   F_\rho''(\alpha)
   =-\frac1{\alpha^3}(1,\rho)\Sigma_{\alpha,\rho}^{-1}
      \binom{1}{\rho}<0.
\]
Hence $F_\rho$ is strictly concave, and its stationary point is characterized by
$Z_{\gamma_\rho(\alpha)}(x_\rho(\alpha))=1$.  Along the unit-partition curve,
\[
   \frac{d\rho_\gamma}{d\gamma}
   =\alpha_\gamma\Var_\gamma\!\bigl(v(J)-\rho_\gamma J\bigr)>0.
\]
\end{lemma}

\begin{proof}
Write $\lambda=\log x$.  The natural-parameter identities are
\[
   \partial_\lambda\log Z_\gamma(e^\lambda)=\E_{\lambda,\gamma}J,
   \qquad
   \partial_\gamma\log Z_\gamma(e^\lambda)=\E_{\lambda,\gamma}v(J).
\]
Their Jacobian is the covariance matrix
\[
   \Sigma_{\alpha,\rho}
   =\begin{pmatrix}
      \Var(J) & \Cov(J,v(J))\\
      \Cov(J,v(J)) & \Var(v(J))
     \end{pmatrix}
\]
under the matched law.  We now vary $\alpha$ with $\rho$ fixed.  The matching equations are
\[
   \E J=\frac1\alpha,
   \qquad
   \E v(J)=\frac\rho\alpha.
\]
Differentiating both equations with respect to $\alpha$ gives
\[
   \Sigma_{\alpha,\rho}
   \binom{\dfrac{d\lambda_\rho}{d\alpha}}
         {\dfrac{d\gamma_\rho}{d\alpha}}
   =-\frac1{\alpha^2}\binom{1}{\rho}.
\]
Now differentiate
\[
   F_\rho(\alpha)
   =\alpha\log Z_{\gamma_\rho(\alpha)}(e^{\lambda_\rho(\alpha)})
    -\lambda_\rho(\alpha)-\rho\gamma_\rho(\alpha).
\]
The chain rule gives
\[
\begin{aligned}
   \frac{dF_\rho}{d\alpha}
   &=\log Z
     +\alpha\left(
       \frac{\partial\log Z}{\partial\lambda}
       \frac{d\lambda_\rho}{d\alpha}
       +
       \frac{\partial\log Z}{\partial\gamma}
       \frac{d\gamma_\rho}{d\alpha}
       \right)
     -\frac{d\lambda_\rho}{d\alpha}
     -\rho\frac{d\gamma_\rho}{d\alpha}\\
   &=\log Z
     +\alpha\left(
       \frac1\alpha\frac{d\lambda_\rho}{d\alpha}
       +\frac\rho\alpha\frac{d\gamma_\rho}{d\alpha}
       \right)
     -\frac{d\lambda_\rho}{d\alpha}
     -\rho\frac{d\gamma_\rho}{d\alpha}\\
   &=\log Z.
\end{aligned}
\]
Here $Z=Z_{\gamma_\rho(\alpha)}(x_\rho(\alpha))$.  Thus the stationary-point equation is exactly $Z=1$.

For the second derivative,
\[
\begin{aligned}
   \frac{d^2F_\rho}{d\alpha^2}
   &=\frac{d}{d\alpha}\log Z\\
   &=\frac1\alpha\frac{d\lambda_\rho}{d\alpha}
     +\frac\rho\alpha\frac{d\gamma_\rho}{d\alpha}\\
   &=\frac1\alpha(1,\rho)
      \binom{\dfrac{d\lambda_\rho}{d\alpha}}
            {\dfrac{d\gamma_\rho}{d\alpha}}\\
   &=-\frac1{\alpha^3}(1,\rho)
      \Sigma_{\alpha,\rho}^{-1}
      \binom{1}{\rho}.
\end{aligned}
\]
The matrix $\Sigma_{\alpha,\rho}$ is positive definite: every canonical law has positive mass on all $j\ge1$, and singularity would force $v(j)$ to be affine in $j$, contrary to $v(j)/j\to\infty$.  Therefore the last quantity is strictly negative, proving strict concavity.

Finally consider the unit-partition curve $Z_\gamma(x_\gamma)=1$ and write
$\lambda_\gamma=\log x_\gamma$.  From
$\log Z_\gamma(e^{\lambda_\gamma})=0$ we obtain
\[
   0
   =\E_\gamma J\,\frac{d\lambda_\gamma}{d\gamma}
    +\E_\gamma v(J),
\]
so
\[
   \frac{d\lambda_\gamma}{d\gamma}=-\rho_\gamma.
\]
More generally, for any function $g$ of at most polynomial growth in $J$ and $v(J)$, differentiation under the sum gives
\[
   \frac{d}{d\gamma}\E_\gamma g(J)
   =\Cov_\gamma\!\left(g(J),
      v(J)+J\frac{d\lambda_\gamma}{d\gamma}\right)
   =\Cov_\gamma\!\left(g(J),v(J)-\rho_\gamma J\right).
\]
Set $H_\gamma(J)=v(J)-\rho_\gamma J$.  Using
$\rho_\gamma=\E_\gamma v(J)/\E_\gamma J$ and the quotient rule,
\[
\begin{aligned}
   \frac{d\rho_\gamma}{d\gamma}
   &=\frac{
      \Cov_\gamma(v(J),H_\gamma)
      -\rho_\gamma\Cov_\gamma(J,H_\gamma)}
      {\E_\gamma J}\\
   &=\frac{\Var_\gamma(H_\gamma)}{\E_\gamma J}
    =\alpha_\gamma\Var_\gamma(H_\gamma)>0.
\end{aligned}
\]
This proves the last assertion.
\end{proof}

Assumption~\ref{ass:profile} also implies $\log v(j)=o(j)$.  Indeed,
\[
 \log v(j+1)-\log v(j)
 =\log\!\left(1+\frac{\Delta v(j)}{v(j)}\right)\longrightarrow0,
\]
so for every $\varepsilon>0$ these increments are eventually bounded by $\varepsilon$, and hence $\log v(j)\le C_\varepsilon+\varepsilon j$.

The derivative formula above is therefore uniform up to the boundary.  By boundary regularity and compactness of a short arc $\gamma\in[\gamma_c-\delta,\gamma_c]$, the laws $q_\gamma$ have a common exponential moment in $J$; together with $\log v(j)=o(j)$, this makes the moments involving $J$ and $v(J)$ in the covariance formula uniformly integrable.  Hence $\alpha_\gamma$, $\rho_\gamma$, and $d\rho_\gamma/d\gamma$ extend continuously to $\gamma_c$.  In particular, because the boundary variance in Lemma~\ref{lem:exp-family} is strictly positive, there are $0<m<M<\infty$ and $\delta>0$ such that
\[
 m\le \frac{d\rho_\gamma}{d\gamma}\le M,
 \qquad \gamma_c-\delta\le\gamma\le\gamma_c.
\]

The next lemma is not needed for the main theorems.  It records the complementary case in which the canonical curve has no finite endpoint, and will be used only in the final example of Section~\ref{sec:examples}.

\begin{lemma}
\label{lem:no-finite-boundary}
Suppose $v$ is increasing, $v(j)/j\to\infty$, and $-\log a_j/v(j)\to\infty$.  Then $Z_\gamma$ is entire for every finite $\gamma$, the unit-partition point $x_\gamma$ exists uniquely, and $\rho_\gamma\to\infty$ as $\gamma\to\infty$.
\end{lemma}

\begin{proof}
We first prove finiteness of $Z_\gamma$ and existence of its unit-partition point.  For each finite $\gamma$, the coefficient assumption gives $\log a_j+\gamma v(j)=-c_jv(j)$ with $c_j\to\infty$; since $v(j)/j\to\infty$, this dominates every linear term in $j$, so $Z_\gamma$ is entire and crosses the level $1$ exactly once.

It remains to prove $\rho_\gamma\to\infty$.  Along the unit-partition curve, Lemma~\ref{lem:exp-family} gives $d\lambda_\gamma/d\gamma=-\rho_\gamma$, where $\lambda_\gamma=\log x_\gamma$.  If $\rho_\gamma\le M$ for all large $\gamma$, then $\lambda_\gamma\ge-M\gamma+O(1)$.  Choose a fixed $j$ with $v(j)/j>M+1$.  Since $q_\gamma(j)\le1$,
\[
  0\ge \log q_\gamma(j)
  =\log a_j+j\lambda_\gamma+\gamma v(j)
  \ge \gamma\{v(j)-Mj\}+O(1)>0
\]
for all sufficiently large $\gamma$, which is impossible.
\end{proof}

\subsection{Local central limit estimates}
\label{sec:lclt}

The one-statistic theorem uses the classical lattice local central limit theorem; see, for example, Petrov~\cite{Petrov1975}.  We record for completeness the form used below.

\begin{lemma}
\label{lem:lclt}
Let $J_1,J_2,\ldots$ be iid positive-integer-valued variables with $\Pp(J=j)>0$ for every $j\ge1$, mean $\mu$, variance $\sigma^2>0$, and an exponential moment near the origin.  If $k\asymp n$ and $n-k\mu=O(1)$, then
\begin{equation}
 \Pp(T_k=n)=\frac{1+o(1)}{\sigma\sqrt{2\pi k}}.
 \label{eq:lclt-centered}
\end{equation}
\end{lemma}

\begin{proof}
This is the standard lattice local central limit theorem; see Petrov~\cite{Petrov1975}.
\end{proof}

For the profile problem the first coordinate remains lattice, since $J$ is integer-valued, while the second is nonlattice under Assumption~\ref{ass:profile}.  The corresponding mixed local limit theorem is classical; see Stone~\cite{Stone1967}.  We record only the fixed-window consequence needed here, with $k/n$ bounded away from $0$ and $1$.

\begin{samepage}
\begin{lemma}
\label{lem:two-lclt}
Suppose Assumption~\ref{ass:profile} holds.  Fix $\varepsilon\in(0,1/2)$, $L<\infty$, and a compact set $K$ of parameters $(x,\gamma)$ with $x>0$ and $Z_\gamma(x)<\infty$, possibly containing $(x_c,\gamma_c)$.  Uniformly for $(x,\gamma)\in K$, $S\in\mathbb R$, and
\[
 \varepsilon\le\frac{k}{n}\le1-\varepsilon,
 \qquad
 |n-k\E_{x,\gamma}J|+|S-k\E_{x,\gamma}v(J)|\le L,
\]
we have, for all sufficiently large $n$,
\[
 \frac{c}{n}\le
 \Pp_{x,\gamma}\bigl(T_k=n,\ |V_k-S|\le1\bigr)
 \le\frac{C}{n},
\]
where $c,C>0$ depend only on the model, $K$, $\varepsilon$, and $L$.
\end{lemma}
\end{samepage}

\begin{proof}
We follow Stone's characteristic-function proof~\cite[Corollary~1, p.~218; proof, p.~221]{Stone1967}.  Write $\zeta=(x,\gamma)$, $X=(J,v(J))$, $m_\zeta=\E_\zeta X$, $\Sigma_\zeta=\Cov_\zeta(X)$, and
$\phi_\zeta(u)=\E_\zeta e^{iu\cdot X}$.  To make Stone's argument uniform over $\zeta\in K$, it is enough to have
\[
 \sup_{\zeta\in K}\E_\zeta\|X\|^3<\infty,
 \qquad
 0<c_\Sigma\le\lambda_{\min}(\Sigma_\zeta)
 \le\lambda_{\max}(\Sigma_\zeta)\le C_\Sigma,
 \qquad
 \sup_{\zeta\in K,\,u\in F}|\phi_\zeta(u)|\le1-\eta_F
\]
for every compact $F\subset([-\pi,\pi]\times\mathbb R)\setminus\{(0,0)\}$, with $\eta_F>0$.

For the first estimate, the consequence $\log v(j)=o(j)$ established above gives $v(j)^3\le e^{\eta j}$ eventually for every $\eta>0$.  Lemma~\ref{lem:partition-geometry} and compactness therefore give, for some $\tau>0$,
\[
 \sup_{\zeta\in K}\E_\zeta e^{\tau J}<\infty,
 \qquad
 \sup_{\zeta\in K}\E_\zeta\|X\|^3<\infty.
\]

For the second, $v$ is not affine and every $q_\zeta(j)$ is positive, so $\Sigma_\zeta$ is positive definite for every $\zeta$.  Its entries depend continuously on $\zeta$, hence compactness of $K$ gives the stated uniform lower and upper eigenvalue bounds.

For the third, suppose $|\phi_\zeta(s,t)|=1$.  Since every value of $J$ has positive probability, $sJ+tv(J)$ must be constant modulo $2\pi$ on $\mathbb N_+$.  Taking first and then second differences gives
\[
 s+t\Delta v(j)\in2\pi\mathbb Z,
 \qquad
 t\Delta^2v(j)\in2\pi\mathbb Z
 \quad( j\ge1).
\]
If $t\ne0$, then $\Delta^2v(j)\to0$ forces $\Delta^2v(j)=0$ for all sufficiently large $j$, so $v$ is eventually affine, contradicting $v(j)/j\to\infty$.  Thus $t=0$, and then $s\in2\pi\mathbb Z$.  In the fundamental domain $s\in[-\pi,\pi]$, equality $|\phi_\zeta(s,t)|=1$ therefore occurs only at $(0,0)$.  Continuity on the compact set $K\times F$ yields the required gap $1-\eta_F$. Under the condition
\[
 |n-k\E_\zeta J|+|S-k\E_\zeta v(J)|\le L,
\]
the leading Gaussian term is bounded above and below by positive multiples of $k^{-1}$.  Applying Stone's theorem leads to
\[
 \frac{c_0}{k}\le
 \Pp_\zeta\bigl(T_k=n,\ |V_k-S|\le1\bigr)
 \le\frac{C_0}{k}.
\]
Since $\varepsilon n\le k\le(1-\varepsilon)n$, the claimed bounds follow.
\end{proof}

The next estimate is a standard consequence of exponential tilting and the lattice local central limit theorem; see, for example, Petrov~\cite{Petrov1975}.  We record the uniform form needed in the proofs below.

\begin{lemma}
\label{lem:offcenter}
Under the assumptions of Lemma~\ref{lem:lclt}, put $\alpha=1/\E J$.  For every $\varepsilon\in(0,1/2)$, there are $c,C>0$ such that, for all sufficiently large $n$,
\begin{equation}
 \Pp(T_k=n)
 \le \frac{C}{\sqrt n}
    \exp\!\left\{-cn\left(\frac{k}{n}-\alpha\right)^2\right\},
 \qquad \varepsilon\le\frac{k}{n}\le1-\varepsilon.
 \label{eq:offcenter-local}
\end{equation}
The constants depend only on $\varepsilon$ and the law of $J$.
\end{lemma}

\begin{proof}
Put $\mu=\E J$, $m=n/k$, and
\[
 \psi(t)=\log\E e^{tJ}.
\]
We first consider $m$ in a fixed small neighborhood of $\mu$.  Since
$\psi'(0)=\mu$ and $\psi''(0)=\Var(J)>0$, there is a unique small
$t=t(m)$ such that $\psi'(t)=m$.  Under the exponentially tilted law
\[
 \frac{\dd\Pp_t}{\dd\Pp}(J)=e^{tJ-\psi(t)},
\]
the sum $T_k$ has mean $km=n$.  Therefore
\[
 \Pp(T_k=n)
 =e^{-kI(m)}\Pp_t(T_k=n),
 \qquad
 I(m)=tm-\psi(t).
\]
For $m$ in this neighborhood, the tilted laws form a compact family with a
uniform exponential moment and variance bounded away from zero.  The local
central limit estimate of Lemma~\ref{lem:lclt} is therefore uniform and gives
\[
 \Pp_t(T_k=n)\le Ck^{-1/2}.
\]
Moreover $I(\mu)=I'(\mu)=0$ and $I''(\mu)=1/\Var(J)>0$, so
$I(m)\ge c(m-\mu)^2$.  Since $k\asymp n$,
\[
 \Pp(T_k=n)
 \le \frac{C}{\sqrt n}
 \exp\!\left\{-c\frac{(n-k\mu)^2}{n}\right\}.
\]

Choose $\delta>0$ small enough that the preceding argument applies whenever
$|m-\mu|\le\delta$.  It remains to treat $|m-\mu|>\delta$.  Since
$k/n\in[\varepsilon,1-\varepsilon]$, the quantity $m=n/k$ ranges over a
fixed compact interval.  For $m\ge\mu+\delta$, choose a fixed sufficiently
small $t_+>0$.  Since $\psi(t)=\mu t+O(t^2)$ as $t\to0$,
\[
 t_+m-\psi(t_+)\ge c_\delta>0.
\]
Likewise, for $m\le\mu-\delta$, a fixed sufficiently small $t_-<0$ gives
$t_-m-\psi(t_-)\ge c_\delta$.  Hence in either case
\[
 \Pp(T_k=n)
 \le \E e^{t_\pm(T_k-n)}
 =\exp\{-k(t_\pm m-\psi(t_\pm))\}
 \le e^{-c_\delta n}.
\]
Since $k/n\in[\varepsilon,1-\varepsilon]$, there is a constant
$B_\varepsilon<\infty$ such that
\[
 \frac{(n-k\mu)^2}{n}\le B_\varepsilon n.
\]
Choose $c_1>0$ so that $c_1B_\varepsilon<c_\delta/2$.  Then, for all
large $n$,
\[
 e^{-c_\delta n}
 \le \frac1{\sqrt n}e^{-c_\delta n/2}
 \le \frac1{\sqrt n}
 \exp\left\{-c_1\frac{(n-k\mu)^2}{n}\right\}.
\]
Thus the same bound obtained near the mean holds throughout the whole
interval $k/n\in[\varepsilon,1-\varepsilon]$.  Finally,
$n-k\mu=-\mu n(k/n-\alpha)$, which gives \eqref{eq:offcenter-local}.
\end{proof}

\subsection{Relative entropy and finite marginals}

For discrete probability laws, write $D(\mu\|\pi)=\sum_z\mu(z)\log(\mu(z)/\pi(z))$, using natural logarithms.  The next two estimates are classical consequences of the chain rule for relative entropy, Pinsker's inequality, and Hoeffding's exponential-moment bound; see Cover and Thomas~\cite[Chapters~2 and~11]{CoverThomas} and Hoeffding~\cite{Hoeffding1963}.  We record for completeness the precise version used below.

\begin{lemma}
\label{lem:entropy-comparison}
Let $\mu$ be an exchangeable law with finite support on $\Nplus^k$, and let $q$ be a probability law with $q(j)>0$ for every $j\ge1$.  Write $\mu^{[r]}$ for the law of the first $r$ coordinates.
\begin{enumerate}[label=(\roman*)]
\item For $1\le r\le k$,
\begin{equation}
 \dTV\!\left(\mu^{[r]},q^{\otimes r}\right)^2
 \le \frac12 D\!\left(\mu^{[r]}\middle\|q^{\otimes r}\right)
 \le \frac{r}{2k}D\!\left(\mu\middle\|q^{\otimes k}\right).
 \label{eq:entropy-marginal}
\end{equation}
\item If $f$ takes values in an interval of length $L$, then
\begin{equation}
 \E_\mu\left|\frac1k\sum_{i=1}^k f(J_i)-\E_q f(J)\right|
 \le L\left(\frac{D(\mu\|q^{\otimes k})+\log2}{2k}\right)^{1/2}.
 \label{eq:entropy-average}
\end{equation}
\end{enumerate}
\end{lemma}

\begin{proof}
For (i), let $\mathcal F_i=\sigma(J_1,\ldots,J_i)$, with $\mathcal F_0$ trivial, and define
\[
 d_i=\E_\mu D\!\left(\Law_\mu(J_i\mid\mathcal F_{i-1})\middle\|q\right).
\]
These are exactly the successive contributions in the chain rule for relative entropy: for every $1\le r\le k$,
\[
 D\!\left(\mu^{[r]}\middle\|q^{\otimes r}\right)
 =\sum_{i=1}^r d_i,
 \qquad
 D\!\left(\mu\middle\|q^{\otimes k}\right)
 =\sum_{i=1}^k d_i.
\]
Thus it remains to compare the average of the first $r$ contributions with the average of all $k$ contributions.  We claim that $d_i$ is nondecreasing.  Indeed, conditionally on $\mathcal F_{i-1}$, the law of $J_{i+1}$ is the mixture, over $J_i$, of the laws conditioned on $\mathcal F_i$.  Convexity of relative entropy therefore gives
\[
 \E_\mu D\!\left(\Law_\mu(J_{i+1}\mid\mathcal F_i)\middle\|q\right)
 \ge
 \E_\mu D\!\left(\Law_\mu(J_{i+1}\mid\mathcal F_{i-1})\middle\|q\right).
\]
By exchangeability, conditionally on $\mathcal F_{i-1}$ the variables $J_i$ and $J_{i+1}$ have the same law, so the right-hand side is $d_i$.  Hence $d_{i+1}\ge d_i$.

It follows that
\[
 \frac1rD\!\left(\mu^{[r]}\middle\|q^{\otimes r}\right)
 =\frac1r\sum_{i=1}^r d_i
 \le\frac1k\sum_{i=1}^k d_i
 =\frac1kD\!\left(\mu\middle\|q^{\otimes k}\right).
\]
Combining this with Pinsker's inequality~\cite[Chapter~11]{CoverThomas} gives
\[
 \dTV\!\left(\mu^{[r]},q^{\otimes r}\right)^2
 \le\frac12D\!\left(\mu^{[r]}\middle\|q^{\otimes r}\right)
 \le\frac{r}{2k}D\!\left(\mu\middle\|q^{\otimes k}\right).
\]

For (ii), put $Q=q^{\otimes k}$ and $X=k^{-1}\sum_i f(J_i)-\E_qf(J)$.
Hoeffding's bound~\cite{Hoeffding1963} gives, for $t\in\mathbb R$,
\[
 \E_Qe^{tX}
 =\left(\E_q e^{(t/k)(f(J)-\E_qf(J))}\right)^k
 \le e^{t^2L^2/(8k)}.
\]
Hence, for $t>0$,
\[
 \E_Qe^{t|X|}
 \le\E_Qe^{tX}+\E_Qe^{-tX}
 \le2e^{t^2L^2/(8k)}.
\]
With $Q_t=e^{t|X|}Q/\E_Qe^{t|X|}$,
\[
 0\le D(\mu\|Q_t)
 =D(\mu\|Q)-t\E_\mu|X|+\log\E_Qe^{t|X|}.
\]
Consequently,
\[
 \E_\mu|X|
 \le\inf_{t>0}\left\{
 \frac{D(\mu\|Q)+\log2}{t}+\frac{tL^2}{8k}
 \right\}
 =L\left(\frac{D(\mu\|Q)+\log2}{2k}\right)^{1/2}.
 \qedhere
\]
\end{proof}

\section{Proofs of the main results}
\label{sec:proofs}

\subsection{Proof of the one-statistic theorem}

\begin{proof}[Proof of \Cref{thm:one-stat}]
At $x=x_*$, $A(x_*)=1$, so
\[
 W^{(a)}_{n,k}=x_*^{-n}\Pp_*(T_k=n).
\]
We first show that the optimizer lies in a fixed interior range, so that Lemma~\ref{lem:offcenter} applies.  Let $\mu=\E_*J$ and $\alpha_*=1/\mu$, and choose $k_n^0=\alpha_*n+O(1)$.  Lemma~\ref{lem:lclt} gives
\[
 x_*^nW^{(a)}_{n,k_n^0}
 =\Pp_*(T_{k_n^0}=n)\ge c n^{-1/2}.
\]
Choose $\varepsilon>0$ with $\alpha_*\in(2\varepsilon,1-2\varepsilon)$.  If $k/n\notin[\varepsilon,1-\varepsilon]$, then $|n-k\mu|\ge c_\varepsilon n$.  Since $q_{x_*}$ has an exponential moment, a Chernoff bound gives
\[
 x_*^nW^{(a)}_{n,k}=\Pp_*(T_k=n)\le e^{-c_\varepsilon n},
\]
so such $k$ cannot maximize for large $n$.

We may therefore apply Lemma~\ref{lem:offcenter} to the maximizing $k_n$:
\[
 \frac{c}{\sqrt n}
 \le \Pp_*(T_{k_n}=n)
 \le \frac{C}{\sqrt n}
 \exp\!\left\{-cn\left(\frac{k_n}{n}-\alpha_*\right)^2\right\}.
\]
Taking logarithms gives
\[
 \left|\frac{k_n}{n}-\alpha_*\right|=O(n^{-1/2}),
\]
which proves the theorem.
\end{proof}

\subsection{Canonical and critical regimes}

\begin{proof}[Proof of \Cref{thm:two-canonical}]
Fix $\gamma\le\gamma_c$ with $\rho=\rho_\gamma>0$.  We first show that the maximizing support lies in a fixed interior range, so that Lemma~\ref{lem:offcenter} can be used.  Let $k_n^0=\alpha_\gamma n+O(1)$.  Since $\E_\gamma J=1/\alpha_\gamma$ and $\E_\gamma v(J)=\rho/\alpha_\gamma$,
\[
 |n-k_n^0\E_\gamma J|+|\rho n-k_n^0\E_\gamma v(J)|=O(1).
\]
Thus the two target values are within bounded distance of their means under $q_\gamma^{\otimes k_n^0}$, and Lemma~\ref{lem:two-lclt} gives
\[
 x_\gamma^n e^{\gamma\rho n}W^{(a,v)}_{n,k_n^0}(\rho)
 \ge c_\gamma n^{-1}.
\]
Choose $\varepsilon>0$ with $\alpha_\gamma\in(2\varepsilon,1-2\varepsilon)$.  For every $k$,
\[
 x_\gamma^n e^{\gamma\rho n}W^{(a,v)}_{n,k}(\rho)
 \le e^{|\gamma|}\Pp_\gamma(T_k=n).
\]
If $k/n\notin[\varepsilon,1-\varepsilon]$, then $|n-k\E_\gamma J|\ge c_\varepsilon n$, and the exponential moment of $q_\gamma$ gives
\[
 x_\gamma^n e^{\gamma\rho n}W^{(a,v)}_{n,k}(\rho)
 \le C_\gamma e^{-c_\varepsilon n}.
\]
Hence every maximizer lies in $[\varepsilon n,(1-\varepsilon)n]$ for all large $n$.

We now apply Lemma~\ref{lem:offcenter} in this interior range.  Since $\E_\gamma J=1/\alpha_\gamma$,
\[
 x_\gamma^n e^{\gamma\rho n}W^{(a,v)}_{n,k}(\rho)
 \le \frac{C_\gamma}{\sqrt n}
 \exp\!\left\{-cn\left(\frac{k}{n}-\alpha_\gamma\right)^2\right\}.
\]
Comparing with the centered lower bound $c_\gamma n^{-1}$ shows first that
\[
 \left|\frac{k_n(\rho)}n-\alpha_\gamma\right|
 =O_\gamma\!\left(\sqrt{\frac{\log n}{n}}\right).
\]
Thus the optimizer is eventually in an arbitrarily small fixed neighborhood of $\alpha_\gamma$.

For (i), assume $\gamma<\gamma_c$.  In that neighborhood, for a competing support $k$ with $k/n=\alpha$, the inverse function theorem applied to
\[
 (\log x,\beta)\longmapsto
 \bigl(\E_{x,\beta}J,\E_{x,\beta}v(J)\bigr)
\]
gives unique nearby parameters $(x,\beta)$ satisfying
$\E_{x,\beta}J=n/k$ and $\E_{x,\beta}v(J)=\rho n/k$.  Thus, for each nearby $k$, we use the canonical law whose two means match that particular support.  The covariance matrix is the derivative of the mean map and is nonsingular by Lemma~\ref{lem:exp-family}.  Lemma~\ref{lem:two-lclt} then gives, uniformly there,
\[
 \log W^{(a,v)}_{n,k}(\rho)
 =nF_\rho(\alpha)-\log n+O_\gamma(1).
\]

At $\alpha=\alpha_\gamma$, the matched parameters are exactly $(x_\gamma,\gamma)$, since $q_\gamma$ matches the two prescribed means and $Z_\gamma(x_\gamma)=1$.  Hence Lemma~\ref{lem:exp-family} gives $F_\rho'(\alpha_\gamma)=0$, and strict concavity yields
\[
 F_\rho(\alpha_\gamma)-F_\rho(\alpha)
 \ge c_\gamma(\alpha-\alpha_\gamma)^2
\]
near $\alpha_\gamma$.

Let $\alpha_n=k_n(\rho)/n$ and $\alpha_n^0=k_n^0/n=\alpha_\gamma+O(n^{-1})$.  Since $k_n(\rho)$ maximizes the microcanonical weight,
\[
 W^{(a,v)}_{n,k_n(\rho)}(\rho)
 \ge W^{(a,v)}_{n,k_n^0}(\rho).
\]
Applying the preceding canonical asymptotic to both supports gives
\[
 nF_\rho(\alpha_n)-\log n+O_\gamma(1)
 \ge nF_\rho(\alpha_n^0)-\log n+O_\gamma(1).
\]
Since $F_\rho'(\alpha_\gamma)=0$ and $\alpha_n^0-\alpha_\gamma=O(n^{-1})$,
$F_\rho(\alpha_n^0)=F_\rho(\alpha_\gamma)+O_\gamma(n^{-2})$.  The preceding comparison therefore gives
\[
 n\bigl(F_\rho(\alpha_\gamma)-F_\rho(\alpha_n)\bigr)=O_\gamma(1).
\]
Combining this with strict concavity gives
\[
 n\left|\frac{k_n(\rho)}n-\alpha_\gamma\right|^2=O_\gamma(1),
\]
which proves (i).

For (ii), take $\gamma=\gamma_c$ and $\rho=\rho_c$.  The preceding interior estimate already gives
\[
 n\left|\frac{k_n(\rho_c)}n-\alpha_c\right|^2=O(\log n),
\]
which is exactly (ii).

For (iii), let $k=k_n(\rho)$ and $\mu=\mu^\rho_{n,k}$.  Parts (i)--(ii) give $k\asymp n$.  On the constraint set $T_k=n$, $|V_k-\rho n|\le1$, the microcanonical and canonical densities are
\[
 \mu(\boldsymbol j)=\frac{\prod_i a_{j_i}}{W^{(a,v)}_{n,k}(\rho)},
 \qquad
 \prod_{i=1}^kq_\gamma(j_i)
 =\prod_i a_{j_i}\,x_\gamma^n e^{\gamma V_k(\boldsymbol j)},
\]
because $Z_\gamma(x_\gamma)=1$.  Hence
\begin{equation}
 \log\frac{\mu(\boldsymbol j)}{\prod_{i=1}^k q_\gamma(j_i)}
 =-\log\!\left(x_\gamma^n e^{\gamma\rho n}W^{(a,v)}_{n,k}(\rho)\right)
   -\gamma\bigl(V_k(\boldsymbol j)-\rho n\bigr).
 \label{eq:profile-likelihood}
\end{equation}
By optimality, the normalized partition function at $k$ is at least its value at the centered support $k_n^0$, and Lemma~\ref{lem:two-lclt} gives
\[
 x_\gamma^n e^{\gamma\rho n}W^{(a,v)}_{n,k}(\rho)
 \ge c_\gamma n^{-1}.
\]
Taking expectation in \eqref{eq:profile-likelihood} therefore gives
\[
 \begin{aligned}
 D\!\left(\mu\middle\|q_\gamma^{\otimes k}\right)
 &\le -\log(c_\gamma n^{-1})+|\gamma|\\
 &=\log n+O_\gamma(1).
 \end{aligned}
\]
Thus the $\log n$ comes exactly from the $n^{-1}$ local-limit probability for the two centered constraints.  Lemma~\ref{lem:entropy-comparison}(i) now yields, for every fixed $r$,
\[
 \dTV\!\left(\mu^{\rho,[r]}_{n,k},q_\gamma^{\otimes r}\right)^2
 \le\frac{r}{2k}(\log n+C_\gamma)
 =O_{\gamma,r}\!\left(\frac{\log n}{n}\right).
\]
This proves (iii).
\end{proof}

\subsection{Supercritical regime}

\begin{proof}[Proof of \Cref{thm:two-supercritical}]
We begin by rewriting the quantity being optimized.  Set
\[
 \mathcal W_{n,k}(\rho)
 :=x_c^n e^{\gamma_c\rho n}W^{(a,v)}_{n,k}(\rho).
\]
Since $x_c^n e^{\gamma_c\rho n}$ is independent of $k$,
\[
 k_n(\rho)\in\argmaxx_{1\le k\le n} W^{(a,v)}_{n,k}(\rho)
 \quad\Longleftrightarrow\quad
 k_n(\rho)\in\argmaxx_{1\le k\le n} \mathcal W_{n,k}(\rho).
\]
Thus all support comparisons may be made with $\mathcal W_{n,k}(\rho)$.

For each fixed $k$, the boundary law $q_c^{\otimes k}$ gives the exact identity
\[
 \mathcal W_{n,k}(\rho)
 =\E_c\!\left[e^{-\gamma_c(V_k-\rho n)}
   \mathbf1_{\{T_k=n,\ |V_k-\rho n|\le1\}}\right].
\]
On the event $|V_k-\rho n|\le1$, the exponential factor is at most $e^{|\gamma_c|}$.  Hence, for every $1\le k\le n$,
\begin{equation}
 0\le \mathcal W_{n,k}(\rho)
 \le e^{|\gamma_c|}\Pp_c(T_k=n).
 \label{eq:supercritical-upper-all-k}
\end{equation}

We next construct one good support.  The idea is to keep almost all coordinates in a bulk close to the critical law $q_c$ and use one exceptional coordinate to carry the excess profile.  A coordinate of size $s$ contributes $s$ to the length and $v(s)$ to the profile.  Relative to a critical bulk carrying profile at rate $\rho_c$ per unit length, its excess contribution is therefore
$v(s)-\rho_cs$.  We choose $s_n$ to be the smallest integer for which
\[
 v(s_n)-\rho_cs_n\ge(\rho-\rho_c)n.
\]
Thus one coordinate of size $s_n$ can account for the entire supercritical excess $(\rho-\rho_c)n$.

After removing this exceptional coordinate, the remaining coordinates must have total length
$N_n=n-s_n$ and total profile $S_n=\rho n-v(s_n)$.  If
\[
 \xi_n=v(s_n)-\rho_cs_n-(\rho-\rho_c)n
\]
denotes the small overshoot in the exceptional coordinate, then
\[
 S_n=\rho_cN_n-\xi_n,
 \qquad
 \frac{S_n}{N_n}=\rho_c-\frac{\xi_n}{N_n}.
\]
The remaining bulk therefore lies just below the critical profile rate.  We choose
$\gamma_n\le\gamma_c$ so that $\rho_{\gamma_n}=S_n/N_n$ and take
$\ell_n=\alpha_{\gamma_n}N_n+O(1)$ to be an integer.  Write $\lambda_\gamma=\log x_\gamma$, so in particular $\lambda_c=\log x_c$.
Restricting the partition sum for $W^{(a,v)}_{n,\ell_n+1}(\rho)$ to configurations whose last coordinate is $s_n$ gives
\begin{equation}
 \begin{aligned}
 \mathcal W_{n,\ell_n+1}(\rho)
 &\ge q_c(s_n)
 \exp\!\left\{N_n(\lambda_c-\lambda_{\gamma_n})
       +(\gamma_c-\gamma_n)S_n\right\}\\
 &\qquad\times
 \E_{\gamma_n}\!\left[
 e^{-\gamma_n(V_{\ell_n}-S_n)}
 \mathbf1_{\{T_{\ell_n}=N_n,\ |V_{\ell_n}-S_n|\le1\}}
 \right].
 \end{aligned}
 \label{eq:one-large-coordinate-bound}
\end{equation}

Under $q_{\gamma_n}^{\otimes\ell_n}$, the target values in
\eqref{eq:one-large-coordinate-bound} differ from their corresponding means by only $O(1)$.  Indeed, since
$\ell_n=\alpha_{\gamma_n}N_n+O(1)$,
\[
 |N_n-\ell_n\E_{\gamma_n}J|
 +|S_n-\ell_n\E_{\gamma_n}v(J)|=O(1),
\]
where the second term uses $\rho_{\gamma_n}=S_n/N_n$.
Also $\gamma_n\to\gamma_c$, because
$\rho_{\gamma_n}=\rho_c-\xi_n/N_n\to\rho_c$.
Lemma~\ref{lem:two-lclt} therefore gives
\[
 \Pp_{\gamma_n}\bigl(T_{\ell_n}=N_n,\ |V_{\ell_n}-S_n|\le1\bigr)
 \ge c_\rho n^{-1}.
\]
On the same event, $|V_{\ell_n}-S_n|\le1$; since $\gamma_n$ stays bounded,
$e^{-\gamma_n(V_{\ell_n}-S_n)}$ is bounded below by a positive constant.
Thus the expectation in \eqref{eq:one-large-coordinate-bound} is at least $c_\rho/n$.  Setting $k_n^0=\ell_n+1$, we obtain
\begin{equation}
 \mathcal W_{n,k_n^0}(\rho)
 \ge \frac{c_\rho}{n}q_c(s_n)e^{R_n},
 \qquad
 R_n=N_n(\lambda_c-\lambda_{\gamma_n})
       +(\gamma_c-\gamma_n)S_n.
 \label{eq:candidate-before-R}
\end{equation}

We now estimate the two factors in \eqref{eq:candidate-before-R}.
First consider $q_c(s_n)$.  Since $v(n)/n\to\infty$, the defining inequality for $s_n$ is satisfied at $s=n$ for all sufficiently large $n$, and hence $s_n\le n$.  By minimality,
$v(s_n-1)-\rho_c(s_n-1)<(\rho-\rho_c)n$, so $v(s_n-1)=O(n)$.
Since $v(j)/j\to\infty$, this forces $s_n=o(n)$.  Using
$\Delta v(j)/v(j)\to0$ then also gives $v(s_n)=O(n)$ and
\[
 0\le \xi_n
 \le \Delta v(s_n-1)+O(1)=o(n).
\]
Finally, because $q_c(j)=a_jx_c^je^{\gamma_cv(j)}$,
\[
 -\log q_c(s_n)
 =\bigl[-\log a_{s_n}-\gamma_cv(s_n)\bigr]-s_n\log x_c.
\]
The first bracket is $o(v(s_n))=o(n)$ by the definition of $\gamma_c$, while the second term is $o(n)$ because $s_n=o(n)$.  Therefore $-\log q_c(s_n)=o(n)$.

It remains to control $R_n$.  Since $S_n/N_n=\rho_{\gamma_n}$ and
$\dd\lambda_\gamma/\dd\gamma=-\rho_\gamma$,
\[
 R_n=-N_n\int_{\gamma_n}^{\gamma_c}
       (\rho_u-\rho_{\gamma_n})\,\dd u.
\]
Near $\gamma_c$, Lemma~\ref{lem:exp-family} gives positive constants $m,M$ such that
$m\le \rho_\gamma'\le M$.  Because
$\rho_c-\rho_{\gamma_n}=\xi_n/N_n$, we have
\[
 m(\gamma_c-\gamma_n)
 \le \frac{\xi_n}{N_n}
 \le M(\gamma_c-\gamma_n).
\]
Thus $\gamma_c-\gamma_n\le \xi_n/(mN_n)$.  Moreover,
$0\le \rho_u-\rho_{\gamma_n}\le M(u-\gamma_n)$ for
$u\in[\gamma_n,\gamma_c]$.  Substituting this bound into the integral representation of $R_n$ gives
\[
\begin{aligned}
 R_n
 &\ge -MN_n\int_{\gamma_n}^{\gamma_c}(u-\gamma_n)\,\dd u\\
 &=-\frac{M}{2}N_n(\gamma_c-\gamma_n)^2\\
 &\ge -\frac{M}{2m^2}\frac{\xi_n^2}{N_n}
 \ge -C_\rho\frac{\xi_n^2}{n},
\end{aligned}
\]
where the last inequality uses $N_n=n-s_n\sim n$.
Substituting the two estimates into \eqref{eq:candidate-before-R} gives
\begin{equation}
 \mathcal W_{n,k_n^0}(\rho)
 \ge \frac{c_\rho}{n}q_c(s_n)
      \exp\!\left\{-C_\rho\frac{\xi_n^2}{n}\right\}
 =e^{-o(n)}.
 \label{eq:supercritical-candidate}
\end{equation}

We now use the two bounds to remove supports near the endpoints.  Choose $\varepsilon>0$ with
$\alpha_c\in(2\varepsilon,1-2\varepsilon)$, and let
\[
 A_n^{\rm out}=\{k: k/n\notin[\varepsilon,1-\varepsilon]\}.
\]
For $k\in A_n^{\rm out}$, the target value $n$ is a linear distance from the mean $k\E_cJ$ under $q_c^{\otimes k}$:
\[
 |n-k\E_cJ|\ge c_\varepsilon n.
\]
The exponential-moment bound used in Lemma~\ref{lem:offcenter} therefore gives, uniformly over all $k\in A_n^{\rm out}$,
\[
 \Pp_c(T_k=n)\le C e^{-c_\varepsilon n}.
\]
Combining this with \eqref{eq:supercritical-upper-all-k},
\[
 \sup_{k\in A_n^{\rm out}}\mathcal W_{n,k}(\rho)
 \le C e^{-c_\varepsilon n}.
\]
On the other hand, the single support $k_n^0$ constructed above satisfies, by \eqref{eq:supercritical-candidate},
\[
 \mathcal W_{n,k_n^0}(\rho)
 \ge \frac{c_\rho}{n}q_c(s_n)e^{-C_\rho\xi_n^2/n}
 =e^{-o(n)}.
\]
Since $e^{-o(n)}$ is eventually much larger than $e^{-c_\varepsilon n}$, none of the supports in $A_n^{\rm out}$ can maximize.  Therefore
\[
 \varepsilon\le \frac{k_n(\rho)}n\le1-\varepsilon
\]
for all sufficiently large $n$.

It remains to compare the supports in the interior interval $[\varepsilon n,(1-\varepsilon)n]$.  For these $k$, Lemma~\ref{lem:offcenter} applies.  Since $\E_cJ=1/\alpha_c$, combining that lemma with \eqref{eq:supercritical-upper-all-k} gives, for every interior competitor,
\begin{equation}
 \mathcal W_{n,k}(\rho)
 \le \frac{C}{\sqrt n}
 \exp\!\left\{-cn\left(\frac{k}{n}-\alpha_c\right)^2\right\},
 \qquad \varepsilon\le k/n\le1-\varepsilon.
 \label{eq:supercritical-upper-interior}
\end{equation}
At the optimizer $k=k_n(\rho)$, comparison with the candidate in \eqref{eq:supercritical-candidate} gives
\begin{equation}
 \frac{c_\rho}{n}q_c(s_n)e^{-C_\rho\xi_n^2/n}
 \le \mathcal W_{n,k}(\rho)
 \le \frac{C}{\sqrt n}
 e^{-cn(k/n-\alpha_c)^2}.
 \label{eq:supercritical-two-sided}
\end{equation}
Taking logarithms yields
\[
 n\left|\frac{k_n(\rho)}n-\alpha_c\right|^2
 \le C_\rho\left(-\log q_c(s_n)+\frac{\xi_n^2}{n}+\log n\right).
\]
Set
\[
 \omega_n(\rho)
 :=\left(\frac{-\log q_c(s_n)+\log n}{n}\right)^{1/2}
   +\frac{\xi_n}{n}.
\]
Then $\omega_n(\rho)\to0$ and
\[
 \left|\frac{k_n(\rho)}n-\alpha_c\right|=O_\rho(\omega_n),
\]
which proves (i).

For (ii), let $k=k_n(\rho)$ and $\mu=\mu^\rho_{n,k}$.  We use the same change-of-measure identity as in the canonical case.  Formula~\eqref{eq:profile-likelihood} with $\gamma=\gamma_c$ says
\[
 D\!\left(\mu\middle\|q_c^{\otimes k}\right)
 =-\log \mathcal W_{n,k}(\rho)
   -\gamma_c\E_\mu[V_k-\rho n].
\]
Since $|V_k-\rho n|\le1$, the second term is bounded. The lower bound in
\eqref{eq:supercritical-two-sided} gives directly
\[
 \begin{aligned}
 -\log \mathcal W_{n,k}(\rho)
 &\le -\log q_c(s_n)+\frac{C_\rho\xi_n^2}{n}+\log n+C_\rho.
 \end{aligned}
\]
Hence
\begin{equation}
 D\!\left(\mu\middle\|q_c^{\otimes k}\right)
 \le C_\rho\left(-\log q_c(s_n)+\frac{\xi_n^2}{n}+\log n\right).
 \label{eq:supercritical-entropy}
\end{equation}
The $\log n$ term is the logarithmic cost of the $n^{-1}$ mixed local-limit factor in the lower bound for the comparison support.  Since $k\asymp n$, Lemma~\ref{lem:entropy-comparison}(i) now gives, for every fixed $r$,
\[
 \dTV\!\left(\mu^{\rho,[r]}_{n,k},q_c^{\otimes r}\right)
 =O_{\rho,r}(\omega_n),
\]
which proves (ii).

It remains to prove (iii).  For a fixed cutoff $M$, put
\[
 I_M=\{i:J_i>M\},\qquad
 p_M=q_c(J>M),\qquad
 u_M=\alpha_c\E_c[v(J)\mathbf1_{\{J>M\}}].
\]
For this fixed $M$, Lemma~\ref{lem:entropy-comparison}(ii) and
\eqref{eq:supercritical-entropy} give
\[
 \E_\mu\left|\frac{|I_M|}{k}-p_M\right|\le C_\rho\omega_n.
\]
Let $f_M(j)=v(j)\mathbf1_{\{j\le M\}}$, and let $L_M$ denote the length of the range of $f_M$.  Since
\eqref{eq:supercritical-entropy} is bounded by $C_\rho n\omega_n^2$ and $k\asymp n$, the same lemma gives
\[
 \E_\mu\left|\frac1k\sum_{i=1}^k f_M(J_i)-\E_cf_M(J)\right|
 \le C_\rho L_M\omega_n.
\]
Moreover,
\[
 \frac1n\sum_{i\in I_M}v(J_i)
 =\frac{V_k}{n}-\frac{k}{n}\,\frac1k\sum_{i=1}^k f_M(J_i).
\]
Since $|V_k-\rho n|\le1$, $|k/n-\alpha_c|\le C_\rho\omega_n$,
$\alpha_c\E_cf_M(J)=\rho_c-u_M$, and
$|\E_cf_M(J)|\le\E_c|v(J)|<\infty$ uniformly in $M$, we obtain
\begin{equation}
 \E_\mu\left|
   \frac1n\sum_{i\in I_M}v(J_i)-(\rho-\rho_c)
 \right|
 \le |u_M|+C_\rho(1+L_M)\omega_n+\frac1n.
 \label{eq:fixed-M-localization}
\end{equation}
Also $p_M\to0$ and $u_M\to0$ as $M\to\infty$.

Choose $M_\ell\uparrow\infty$ so that
\[
 p_{M_\ell}+|u_{M_\ell}|\le \frac1\ell.
\]
After $M_\ell$ is chosen, $L_{M_\ell}<\infty$.  Since $\omega_n\to0$, we may choose integers $n_\ell\uparrow\infty$ such that, for every $n\ge n_\ell$,
\[
 C_\rho\omega_n\le\frac1\ell,
 \qquad
 C_\rho(1+L_{M_\ell})\omega_n+\frac1n\le\frac1\ell.
\]
Define $M_n=M_\ell$ whenever $n_\ell\le n<n_{\ell+1}$, and let $I_n=\{i:J_i>M_n\}$.  Then on this block,
\[
 \E_\mu\frac{|I_n|}{k}\le\frac{2}{\ell},
 \qquad
 \E_\mu\left|
   \frac1n\sum_{i\in I_n}v(J_i)-(\rho-\rho_c)
 \right|\le\frac{2}{\ell}.
\]
Since $\ell\to\infty$ as $n\to\infty$, these quantities vanish and $M_n\to\infty$.  This proves (iii).
\end{proof}

\section*{Acknowledgments}
The author thanks Stefan Steinerberger, Shannon Starr, and Persi Diaconis for their helpful feedback.

\bigskip
\begin{flushleft}
\small
\textsc{Department of Mathematics}\\
University of Alabama at Birmingham\\
Birmingham, Alabama, USA\\
\href{mailto:ottolini@uab.edu}{\texttt{ottolini@uab.edu}}
\end{flushleft}


\small
\footnotesize
\begin{thebibliography}{99}
\setlength{\itemsep}{0pt}
\setlength{\parsep}{0pt}

\bibitem{AvivSmudge}
A.~J. Aviv, K.~Gibson, E.~Mossop, M.~Blaze, and J.~M. Smith.
Smudge attacks on smartphone touch screens.
In \emph{4th USENIX Workshop on Offensive Technologies (WOOT 10)}, 2010.

\bibitem{ArmendarizLoulakisHeavy}
I.~Armend\'ariz and M.~Loulakis.
Conditional distribution of heavy tailed random variables on large
deviations of their sum.
\emph{Stochastic Process. Appl.}, 121:1138--1147, 2011.

\bibitem{CarayolRotondo}
A.~Carayol and P.~Rotondo.
Efficient uniform sampling of surjections via their profiles.
\emph{arXiv:2605.24068}, 2026.

\bibitem{ChatterjeeSoliton}
S.~Chatterjee.
Invariant measures and the soliton resolution conjecture.
\emph{Comm. Pure Appl. Math.}, 67:1737--1842, 2014.

\bibitem{Chatterjee}
S.~Chatterjee.
A note about the uniform distribution on the intersection of a simplex and
a sphere.
\emph{J. Topol. Anal.}, 9:717--738, 2017.

\bibitem{ChakrabortyOng}
S.~Chakraborty and S.~H.~Ong.
A COM--Poisson-type generalization of the negative binomial distribution.
\emph{Commun. Statist. Theory Methods}, 45:4117--4135, 2016.

\bibitem{ConwayMaxwell}
R.~W. Conway and W.~L. Maxwell.
A queuing model with state dependent service rates.
\emph{Journal of Industrial Engineering}, 12:132--136, 1962.

\bibitem{CoverThomas}
T.~M. Cover and J.~A. Thomas.
\emph{Elements of Information Theory}.
Wiley, Hoboken, NJ, 2nd edition, 2006.

\bibitem{Csiszar1984}
I.~Csisz\'ar.
Sanov property, generalized $I$-projection and a conditional limit theorem.
\emph{Ann. Probab.}, 12:768--793, 1984.

\bibitem{DiaconisFreedmanDozen}
P.~Diaconis and D.~Freedman.
A dozen de Finetti-style results in search of a theory.
\emph{Ann. Inst. H. Poincar\'e Probab. Statist.}, 23(S2):397--423, 1987.

\bibitem{DiaconisFreedmanExp}
P.~Diaconis and D.~A. Freedman.
Conditional limit theorems for exponential families and finite versions of
de Finetti's theorem.
\emph{J. Theoret. Probab.}, 1:381--410, 1988.

\bibitem{GrosskinskySchutzSpohn}
S.~Gro\ss kinsky, G.~M. Sch\"utz, and H.~Spohn.
Condensation in the zero range process: stationary and dynamical properties.
\emph{J. Stat. Phys.}, 113:389--410, 2003.

\bibitem{JansonAllocations}
S.~Janson.
Simply generated trees, conditioned Galton--Watson trees, random allocations
and condensation.
\emph{Probab. Surv.}, 9:103--252, 2012.

\bibitem{KolchinEtAl}
V.~F. Kolchin, B.~A. Sevast'yanov, and P.~V. Chistyakov.
\emph{Random Allocations}.
V.~H. Winston \& Sons, Washington, DC, 1978.

\bibitem{LewisPfisterSullivan1994}
J.~T. Lewis, C.-E. Pfister, and W.~G. Sullivan.
The equivalence of ensembles for lattice systems: some examples and a counterexample.
\emph{J. Stat. Phys.}, 77:397--419, 1994.

\bibitem{LewisPfisterSullivan}
J.~T. Lewis, C.-E. Pfister, and W.~G. Sullivan.
Entropy, concentration of probability and conditional limit theorems.
\emph{Markov Process. Related Fields}, 1:319--386, 1995.

\bibitem{Mezo}
I.~Mez\H{o}.
Asymptotics of the modes of the ordered Stirling numbers.
\emph{arXiv:1504.06970}, 2015.

\bibitem{Nam2020}
K.~Nam.
Large deviations and localization of the microcanonical ensembles given by multiple constraints.
\emph{Ann. Probab.}, 48:2525--2564, 2020.

\bibitem{Petrov1975}
V.~V. Petrov.
\emph{Sums of Independent Random Variables}.
Springer, Berlin, 1975.

\bibitem{VanCampenhoutCover}
J.~M. Van Campenhout and T.~M. Cover.
Maximum entropy and conditional probability.
\emph{IEEE Trans. Inform. Theory}, 27:483--489, 1981.

\bibitem{SzavitsPRL}
J.~Szavits-Nossan, M.~R. Evans, and S.~N. Majumdar.
Constraint-driven condensation in large fluctuations of linear statistics.
\emph{Phys. Rev. Lett.}, 112:020602, 2014.

\bibitem{SzavitsJPA}
J.~Szavits-Nossan, M.~R. Evans, and S.~N. Majumdar.
Condensation transition in joint large deviations of linear statistics.
\emph{J. Phys. A}, 47:455004, 2014.

\bibitem{Touchette2015}
H.~Touchette.
Equivalence and nonequivalence of ensembles: thermodynamic, macrostate, and measure levels.
\emph{J. Stat. Phys.}, 159:987--1016, 2015.

\bibitem{PitmanCSP}
J.~Pitman.
\emph{Combinatorial Stochastic Processes}.
Lecture Notes in Mathematics 1875. Springer, Berlin, 2006.

\bibitem{BerestyckiPitman}
N.~Berestycki and J.~Pitman.
Gibbs distributions for random partitions generated by a fragmentation process.
\emph{J. Stat. Phys.}, 127:381--418, 2007.

\bibitem{BenderEnumeration}
E.~A. Bender.
Asymptotic methods in enumeration.
\emph{SIAM Rev.}, 16:485--515, 1974.

\bibitem{BorgaDasMukherjeeWinkler}
J.~Borga, S.~Das, S.~Mukherjee, and P.~Winkler.
Large deviation principle for random permutations.
\emph{Int. Math. Res. Not. IMRN}, 2024:2138--2191, 2024.

\bibitem{DenHollanderDense}
F.~den Hollander, M.~Mandjes, A.~Roccaverde, and N.~J. Starreveld.
Ensemble equivalence for dense graphs.
\emph{Electron. J. Probab.}, 23:paper no.~12, 26 pp., 2018.

\bibitem{KenyonKralRadinWinkler}
R.~Kenyon, D.~Kr\'al', C.~Radin, and P.~Winkler.
Permutations with fixed pattern densities.
\emph{Random Structures Algorithms}, 56:220--250, 2020.

\bibitem{MukherjeePermExp}
S.~Mukherjee.
Estimation in exponential families on permutations.
\emph{Ann. Statist.}, 44:853--875, 2016.

\bibitem{PemantleWilsonMelczer}
R.~Pemantle, M.~C. Wilson, and S.~Melczer.
\emph{Analytic Combinatorics in Several Variables}.
Cambridge University Press, 2nd edition, 2024.

\bibitem{SquartiniEtAl}
T.~Squartini, J.~de Mol, F.~den Hollander, and D.~Garlaschelli.
Breaking of ensemble equivalence in networks.
\emph{Phys. Rev. Lett.}, 115:268701, 2015.

\bibitem{Temme}
N.~M. Temme.
Asymptotic estimates of Stirling numbers.
\emph{Stud. Appl. Math.}, 89:233--243, 1993.

\bibitem{Stone1967}
C.~Stone.
On local and ratio limit theorems.
In \emph{Proceedings of the Fifth Berkeley Symposium on Mathematical Statistics and Probability}, Vol.~II, Part~II, pages 217--224. University of California Press, 1967.

\bibitem{EckGeyer}
D.~J. Eck and C.~J. Geyer.
Computationally efficient likelihood inference in exponential families when the maximum likelihood estimator does not exist.
\emph{arXiv:1803.11240}, 2020.

\bibitem{Hoeffding1963}
W.~Hoeffding.
Probability inequalities for sums of bounded random variables.
\emph{J. Amer. Statist. Assoc.}, 58:13--30, 1963.

\bibitem{BialasBurdaJohnston}
P.~Bia\l as, Z.~Burda, and D.~Johnston.
Phase diagram of the mean field model of simplicial gravity.
\emph{Nucl. Phys. B}, 542:413--424, 1999.

\bibitem{ZamparoRenewal}
M.~Zamparo.
Large deviations in discrete-time renewal theory.
\emph{Stochastic Process. Appl.}, 139:80--109, 2021.

\bibitem{ZamparoModels}
M.~Zamparo.
Large deviations in renewal models of statistical mechanics.
\emph{J. Phys. A: Math. Theor.}, 52:495004, 2019.

\end{thebibliography}
\end{document}